\RequirePackage{fix-cm}
\documentclass[smallcondensed]{svjour3}     
\smartqed  
\usepackage{latexsym}
\usepackage{ragged2e}
\usepackage{amssymb,amsfonts,fancyhdr}
\usepackage{picinpar,subfigure}
\usepackage{float}
\usepackage{pstricks}
\usepackage{fancyvrb}
\usepackage{cite}
 \usepackage{mathrsfs}
\usepackage{multirow}
\usepackage{rotating,makecell}
\usepackage{caption}
\usepackage{ulem}

\usepackage{amsmath}
  \usepackage{paralist}
  \usepackage{graphics} 
  \usepackage{epsfig} 
\usepackage{graphicx}  \usepackage{epstopdf}
 \usepackage[colorlinks=true]{hyperref}

\newcommand{\dd }{\mathrm{d}}

 \journalname{Communications on Applied Mathematics and Computation}
\begin{document}

\title{Radon Measure Solutions to Riemann Problems for Isentropic Compressible Euler Equations of Polytropic Gases \thanks{This work was supported by the National Natural Science Foundation of China  under Grants No. 11871218, No. 12071298, and by the Science and Technology Commission of Shanghai Municipality  under Grant No. 18dz2271000.}\\ \footnotesize{Dedicated to Professor Tong Zhang on the occasion of his 90th birthday}
}

\titlerunning{Radon Measure Solutions to Riemann Problems}        

\author{
Yunjuan Jin        \and Aifang Qu \and Hairong Yuan
}

\authorrunning{Jin, Qu and Yuan} 

\institute{Yunjuan Jin \at
              Center for Partial Differential Equations, School of Mathematical Sciences, East China Normal University, Shanghai 200241, China   \\
              \email{52185500019@stu.ecnu.edu.cn}
           \and
          Aifang Qu \at
              Department of Mathematics, Shanghai Normal University, Shanghai  200234,  China\\
    \email{afqu@shnu.edu.cn}
    \and
 Hairong Yuan * (Corresponding author) \at
              School of Mathematical Sciences and Shanghai Key Laboratory of Pure Mathematics and Mathematical  Practice, East China Normal University, Shanghai 200241, China  \\
              \email{hryuan@math.ecnu.edu.cn}
}

\date{Received: date / Accepted: date}

\maketitle

\begin{abstract}
We solve the Riemann problems for isentropic compressible Euler equations of polytropic gases in the class of Radon measures, and the solutions admit the concentration of mass. It is found that, under the requirement of satisfying the over-compressing entropy condition:  (i) there is a unique delta shock solution, corresponding to the case that has two strong classical Lax shocks;  (ii) for the initial data that the classical Riemann solution contains a shock wave and a rarefaction wave, or two shocks with one being weak, there are infinitely many solutions, each consists of a delta shock and a rarefaction wave;  (iii)  there is no delta shocks for the case that the classical entropy weak solutions consist only of rarefaction waves. These solutions are self-similar. 
Furthermore, for the generalized Riemann problem with mass concentrated initially at the discontinuous point of initial data, there always exists a unique delta shock for at least a short time. It could be prolonged to a global solution.  Not all the solutions are self-similar due to the initial velocity of the concentrated point-mass (particle). Whether the delta shock solutions constructed satisfy the over-compressing entropy condition is clarified. This is the first result on the construction of singular measure solutions to the compressible Euler system of polytropic gases, that is strictly hyperbolic, and whose characteristics are both genuinely nonlinear. We also discuss possible physical interpretations and applications of these new solutions.

\end{abstract}

\section{Introduction}\label{sec1}

We consider the following one-space-dimensional unsteady isentropic compressible Euler equations of gas dynamics
\begin{equation}\label{1.1}
\left \{
\begin{split}
& \rho_{t}+(\rho u)_{x}=0,\\
&(\rho u)_{t}+(\rho u^{2}+p(\rho))_{x}=0.\\
\end{split}\right.
\end{equation}
Here  $t\ge0$ and $x\in\mathbb{R}$ are the independent time and space variables.
The unknown functions $\rho(x
,t)\geq 0,\  p(x,t)\geq 0,\ u(x,t)$ denote the mass density,  pressure, and velocity of the gas respectively. For polytropic gases,  the state equation $p=p(\rho)$  takes the form $p(\rho)\doteq\rho^{\gamma}$, with $\gamma\ge1$ being the adiabatic exponent.

It is well-known that in 1860, Riemann \cite{Riemann} pioneered the research of discontinuous solutions of hyperbolic conservation laws by studying the Cauchy problem for \eqref{1.1}, with initial data of the form
\begin{equation}\label{1.01}
  (u(x,0), \rho(x,0))=\begin{cases}
                          (u_1,\rho_1), & \mbox{if~$x<0$},\\
                          (u_2,\rho_2), & \mbox{if~$x>0$},
                        \end{cases}
\end{equation}
where $\rho_{1,2}>0$ and $u_{1,2}$ are all constants. He constructed piecewise smooth solutions  consisting of constant states, shock waves, and/or rarefaction waves. He also proposed an entropy condition (which is a special case of the Lax condition \cite[(8.3.1) in p.240]{Da}), to guarantee uniqueness of solutions in the class of piecewise-smooth functions, and the corresponding solutions are called {\it entropy solutions}. It turns out that this problem and its entropy solutions are fundamental for the mathematical theory and numerical computation of hyperbolic partial differential equations and their applications to aerodynamics and engineering. Hence, Cauchy problems like \eqref{1.1}\eqref{1.01} with piecewise constant initial data are named as ``Riemann problems". In this paper, we will further look for solutions to \eqref{1.1}\eqref{1.01} in the class of Radon measures, which admit concentration of mass in the $(x,t)$-physical space.
For convenience of comparison, we list the entropy solutions obtained by Riemann in the Appendix of this paper. One may also refer to \cite{Chang1989The,smoller1994shock} for more details on the classical theory.

Why do we study Radon measure solutions to the Riemann problem \eqref{1.1}\eqref{1.01}? To answer this question, we briefly review, best to our knowledge, some development on the mathematical theory of hyperbolic conservation laws and compressible Euler equations.

Up to now, by using Riemann problems as building blocks, a complete mathematical theory on well-posedness and asymptotic behavior of entropy weak solutions of strictly hyperbolic conservation laws in a single space variable with genuinely nonlinear or linearly degenerate characteristics has been established for initial values with small total variations (cf. \cite{bressan,Da,holden}). However, difficulties related to concentration may arise, due to failure of strict hyperbolicity, or problems associated with large initial data (cf. Remark 3 in \cite{bressan2011}  and \cite[Section 9.6]{Da}). Certain Riemann problems for hyperbolic conservation laws do not have  solutions in the class of locally Lebesgue-integrable functions, when the two constant states are not close; for a number of models where strict hyperbolicity fails or all the characteristics are linearly degenerate, the solutions with bounded initial value may lie in the larger space of measures (see, for instance,    \cite{brenier2005solutions,Daw2016Shadow,guo2010chaplygin,Jiang2016Developing,Jiang2021The,Jin2019On, Keyfitz1999Conservation,Liu2020Riemann, Nedeljkov2004Delta,shen2011globalstructure,Shen2016The,Tan1994Two,Tan1994Delta,Yang2012New}).

A non-strictly hyperbolic system requiring measure solutions was firstly reported by Korchinski \cite{Korchinski1977Solution} in 1977:
\begin{align}\label{eq12}
  u_t+(\frac12u^2)_x=0,\qquad v_t+(\frac12 uv)_x=0.
\end{align}
Motivated by numerical results, he extended the class of admissible solutions to that of ``$\delta$-solution" --- ``a distribution which is the sum of a function, continuous except on a finite set of curves, and one or several {\it generalized $\delta$-functions}, and which satisfies the integral form of the partial differential equations (in conservation form), and the initial conditions. If no $\delta$-function term appears, then to be admissible the solution must be of classical weak solution" (\cite[p.30]{Korchinski1977Solution}).
He solved Riemann problems for \eqref{eq12} in this class when $u$ develops shocks and then $v$ contains Dirac measures supported on space-time lines. Notice that the ``generalized delta functions" were introduced by him to overcome the difficulty of multiplication of a Dirac measure and a discontinuous function appearing in the term $\frac12 uv$. This is always a key point to understand  how a measure could satisfy a nonlinear partial differential equation in a reasonable sense.  We remark that the Radon measure solutions we seek to \eqref{1.1}\eqref{1.01} still lie in the class defined by Korchinski, while we use the Radon-Nikodym derivatives of measures instead of generalized delta functions to avoid any confusion.

Being not aware of Korchinski's unpublished work (see \cite[p.2]{Tan1994Delta}), in 1990s, Tan and Zhang reported in \cite[p.247]{Tan1994Two} and \cite{Chang1991Two} their discovery of necessity of introducing Dirac measures when solving the two-dimensional Riemann problems for the system
\begin{align}\label{eq13}u_t+(u^2)_x+(uv)_y=0, \qquad v_t+(uv)_x+(v^2)_y=0,\end{align}
whose hyperbolicity fails on a curve in the $(u,v)$-state space. They also coined the now well-accepted term ``delta shock" ($\delta$-shock, see \cite[p.247]{Tan1994Two} and \cite[p.289]{Da}). It is remarkable that, like Riemann, who derived the Rankine-Hugoniot conditions and entropy condition of \eqref{1.1} without a rigorous mathematical definition of weak solutions,  Tan and Zhang, with their great intuition, obtained the generalized Rankine-Hugoniot conditions
\cite[p.245]{Tan1994Two} and the $\delta$-entropy condition \cite[p.246]{Tan1994Two} --- ``all of the characteristic curves on both sides of the discontinuity curve are incoming at every point on the discontinuity", also without a definition of measure solutions to \eqref{eq13}. They also showed that delta shock is well-posed in the sense that if the Riemann initial data admit a classical solution, then they admit no delta shock solution, and vice-versa (p.247). Tan, Zhang, and Zheng justified delta shock of \eqref{eq12} by viscous approximation in \cite{Tan1994Delta}.

Keyfitz and Kranzer had studied Riemann problems for the system
\begin{align}\label{eq14}u_t+(u^2-v)_x=0, \qquad v_t+(\frac13 u^3-u)_x=0,\end{align}
which is strictly hyperbolic and genuinely nonlinear \cite[p.421]{Keyfitz1995Spaces}. They found that for some ``large data", there is no classical Riemann solutions.
By constructing approximate solutions via the artificial self-similar viscosity approximation developed by Tup\u{c}iev \cite{Tupciev1973The} and Dafermos \cite{Dafermos1973Solution}, they showed that the approximate solutions have a limit in some topology, and called  such a limit a singular shock. It is not clear how to give a sense in which the limit measures satisfy the equations, perhaps due to the nonlinearity in $u$.

Since the equations \eqref{eq12}-\eqref{eq14} are somewhat artificial, the more physically significant pressureless Euler equations (i.e., \eqref{1.1} with $p(\rho)\equiv0$) attract great attention to study measure solutions. It is indicated to explain the formation of large-scale structure in the universe.  The characteristics of this system are coincident, thus it is not strictly hyperbolic.
{ Due to its special structure, its measure solutions could be defined through distribution, with Lebesgue measure of physical space replaced by the mass measure,} 
see, for example \cite[Section 3.3]{lizhangyang}, and the generalized Rankine-Hugoniot conditions could be derived rigorously.  Bouchut \cite{Bouchut1994On} proposed {\red the} Borel measure solutions and showed existence of Riemann solutions containing delta shocks by viscosity approximation. Then E, Rykov, and Sinai  \cite{E1996Generalized} constructed a global weak solution to the Cauchy problem with density being a Radon measure, by making use of generalized variational method. They also proved that a convex entropy-entropy flux pair is not sufficient to single out all non-physical solutions for the pressureless flow. Cheng, Li, and Zhang \cite{chenglizhang1997} and Li, Warnecke \cite{liw2003} firstly established well-posdness of measure solutions by the method of generalized characteristics.
Wang, Huang, and Ding \cite{ding1997cauchyproblemoftransportationequations}
established the existence for general velocity by employing  Lebesgue-Stieltjes integral and the potential function similar to \cite{E1996Generalized}.
Huang and Wang proved  uniqueness of the measure solution under the Oleinik entropy condition and the energy condition for initial data belonging to the space of Radon measures \cite{Huang2001Well}. 

The pressureless Euler equations also describe the dynamics of sticky particles, which provides a different approach to construct measure solutions through interacting discrete particles and optimal transport methods. Since these methods and the motivations are quite different from the present paper, we will not review them in detail, but just recommend   \cite{Brenier1998Sticky,Natile2009A,Cavalletti2015A,Nguyen2015One,Hynd2018Probability} to the interested readers.

Except the aforementioned general mathematical theories, to justify the mathematical concept and physical implications of delta-shocks, there are numerous literatures on the analysis of approximate problems and convergence of their solutions. One specific direction of study is the vanishing viscosity method as mentioned above \cite{Bouchut1994On,Boudin2000A,Tan1994Delta,Li2001Delta}. The other is on flux approximation. Li \cite{Li2001Note} firstly proved that delta shock appears as the temperature goes to zero for the Riemann problem \eqref{1.1}\eqref{1.01}, if it admits two shocks. Chen and Liu  \cite{Chen2003Formation,Chen2004Concentration} justified similar results for the vanishing pressure limit. See also \cite{Hu1997A,Sheng1999The,Yang1999Riemann,Ercole2000Delta,Cheng2008Two} for more related works.

The Euler equations \eqref{1.1} of the Chaplygin gas (i.e. $p(\rho)=-1/\rho$) provide a prototype of a strictly hyperbolic system of conservation laws with linearly degenerate characteristics that admits delta shocks. Brenier \cite{brenier2005solutions} investigated the Riemann problem and constructed delta shock  with a mass concentration located in the moving point of discontinuity.  Dozens of articles were devoted to the study of delta shocks of Chaplygin gas and modified Chaplygin gases, see, for example, \cite{guo2010chaplygin} { and those cited it}.

There are many other ways to define measure solutions. For example, Bouchut and James \cite{BJ1997} employed a duality method to solve linear transport equations for BV functions, and the linear continuity equations for measures. For equations in nonconservative form, one may also consult \cite{LeFloch1990} for ideas on definitions of products of measures and discontinuous functions.  The methods of shadow waves, split $\delta$-functions, Colombeau's generalized functions, as well as Sarrico's $\alpha$-product framework,  could be found in  \cite{Colombeau1992Multiplication,Nedeljkov2002Unbounded,Nedeljkov2004Delta, Nedeljkov2008Interactions,Paiva2020Formation}. From these works, we notice that an underlying philosophy is this: {\it to solve a Riemann problem not solvable in the usual sense of functions with minimal waves, which means that one has an over-determined (usually algebraic) problem. Then one adds some extra structure in the solution, thus has more freedom to get a solvable system.} For example, introducing a delta shock in the flow field brings some weights which are time-dependent functions to be solved; the shadow wave method adds extra narrow regions in the Riemann solution, such as two constant intermediate states for a $2\times 2$ system, rather than one for a classical Riemann solution, etc. It is obvious that adding more freedom, there will be more solutions. So in the construction of delta shocks, we shall always try to {\it minimize the total number of waves}. The crucial point left is how to give a reasonable sense, i.e., definition, that the extra structure satisfies the differential equations. To judge whether a definition of (measure, or generalized) solution is acceptable and useful, apart from its mathematical preciseness and theoretical consistency, from applied mathematics points of view, one also wishes to obtain from the definition some  significant results that were demonstrated by physical experiments or engineering practice.

To tackle real physical problems, one usually needs to study initial-boundary value problems, rather than the Cauchy problems, for which all the works mentioned above were devoted to. In 2018, motivated by the paper \cite{hu2018} authored by Hu, Qu, Yuan, and Zhao started to study the problem of hypersonic-limit flow passing straight wedges \cite{Qu2020Hypersonic}. It turns out that under suitable scaling, the hypersonic limit is exactly the vanishing pressure limit, and the hypersonic-limit flow is actually the well-studied pressureless Euler flow. The authors proposed a rigorous definition of Radon measure solutions of a boundary value problem for the two-space-dimensional stationary compressible Euler equations with general state equations. The new definition employs the Radon-Nikodym derivatives of absolutely continuous measures, rather than the integration of velocity with respect to the mass density measure used in the previous works, thus eliminates the confusion that whether a discontinuous function is integrable with respect to a general Borel measure. The authors also showed that as the Mach number of upcoming flow increases to infinity, the Lax shock-front ahead of the wedge approaches the wedge, and the classical Riemann solutions converge weakly as measures to a singular Radon measure solution.  What's more, as a by-product of the limiting Radon measure solution, for which mass concentrates on the wedge, one obtains the Newton's sine-squared pressure law, which is a fundamental formula for hypersonic aerodynamics \cite[Section 3.2]{Anderson2006Hypersonic}.

It turns out that the definition presented in \cite{Qu2020Hypersonic} is rather flexible, as it could be used to deal with quite different and more difficult problems. By the definition of Radon measure solutions, we proved the Newton-Busemann law for hypersonic flow passing curved wedges and cones, and obtained formulas not known before for pressure distributions on bodies in hypersonic flows\cite{quyuan2021}; we studied interactions of delta shocks leaving finite obstacles and discovered the extinction of delta shocks, which was not reported before, see \cite{Jin2019On,Jin2021Radon,Qu2020Radon,Qu2020Hypersonic,Qu2021Radon}. In \cite{Qu2020Measure,Qu2020High}, the authors also studied the high-Mach number limit of piston problems for the Euler equations of polytropic gases and the Chaplygin gas, and measure solutions with density concentrated on pistons were constructed.

The studies of hypersonic-limit flows and piston problems show that all the previous works on measure solutions are not only of mathematical curiosity, but also closely connected to significant physical phenomena and hypersonic aerodynamics. It is somewhat out of expectation that, to our knowledge, priori to \cite{Qu2020Hypersonic,Jin2019On,Qu2020Radon,Qu2020Measure,Qu2020High},  there is no research paper on measure solutions to initial-boundary-value problems for hyperbolic conservation laws. (It is noted that until recently, Neumann and Oberguggenberger \cite{Neumann2021Initial} studied measure-valued solutions to an initial-boundary-value problem for the one-space-dimensional pressureless Euler equations, by using the method of \cite{Huang2001Well}.)

Being confident with the above concept of Radon measure solutions, we turn to the classical Riemann problem \eqref{1.1}\eqref{1.01} for polytropic gases, wondering whether it has delta shock solutions. The point is that, to our knowledge, no delta shock solution has ever been constructed for a strictly hyperbolic system of conservation laws whose characteristic families are all genuinely nonlinear. The existence or non-existence will definitely help us to understand better the physics of delta shocks, as well as to what extent the compressible Euler equations are valid.

Theorem \ref{thm3.1} in this paper (see Section \ref{s3.1}) shows that, roughly speaking: ~(i)~for most of the Riemann data that produce two strong classical Lax shocks, there is a unique admissible delta-shock solution; ~(ii)~for the Riemann data that the classical Riemann solution contains a shock wave and a rarefaction wave, or two shocks with one being rather weak, there are infinitely many solutions, each of which consists of a delta shock and a rarefaction wave;  ~(iii)~there is no admissible delta shocks for the case that the classical entropy weak solutions  consist only of rarefaction waves. These new solutions satisfy the over-compressing entropy condition (see \eqref{2.66}), and the delta shocks are straight lines in the $(x,t)$-physical plane.

We also solve the generalized (or singular) Riemann problem for \eqref{1.1}, which means the initial data are
\begin{equation}\label{1.2}
\left \{
\begin{split}
&\varrho_0=\rho_{1}\mathcal{L}^{1}\lfloor\{x<0\}+\rho_{0}\delta_{\{x=0\}}+\rho_{2}\mathcal{L}^{1}\lfloor\{x>0\},\\
&u_0(x)=u_{1}\mathsf{I}_{\{x<0\}}+u_{0}\mathsf{I}_{\{x=0\}}+u_{2}\mathsf{I}_{\{x>0\}}, \\
\end{split} \right.
\end{equation}
where $u_0(x)$ is considered as a function measurable with respect to the Radon measure $\varrho_0$, which is the sum of a  Dirac measure $\delta_{\{x=0\}}$ supported at the origin, and the standard  Lebesgue measure $\mathcal{L}^{1}$ on $\mathbb{R}$, restricted to the positive/negative real axis. We use $\mathsf{I}_{A}$ to denote the indicator function of a set $A$, and $m\lfloor A$ is the measure obtained by restricting a measure $m$ on an $m$-measurable set $A$. See Theorem \ref{thm4.1} for the complete results, which in particular imply local existence of a delta shock if the initial concentration $\rho_0$ is positive.

Since the classical Riemann solutions (see Appendix) had been thoroughly examined mathematically and physically for over a century, one would argue that the obtained delta shocks are extraneous, as it is well-known that there are too many weak
solutions to conservation laws but physically meaningless. However, in a conversation with Professor Jiequan Li, he suggested a natural connection between delta shocks and the free pistons which could absorb or release gases while the gases pushing it at the two sides. The identification of a delta shock as a free piston had been verified for the pressureless gas \cite{Gao2021Free}. In Remark \ref{rmk4new}, we show that from the generalized Rankine-Hugoniot conditions of delta shocks, one could derive the movement of a classical piston that neither absorbs nor releases gas all the time. The key observation is, {\it using the concept of delta shocks, we could transform a complicated solid-fluid interaction problem, which consists of coupled initial boundary value problems (modeling the gas) and moving boundaries (describing the pistons), to a single Cauchy problem, and treat the large scale of fluids and small scale of particles/pistons in a unified way. This new approach may significantly reduce expenses on analysis and computations.}  Also, our results indicate that a delta shock (i.e., a suitable piston absorbs gases) provides a way to eliminate shock waves in the polytropic gases, which helps to smooth up flow fields, and might have some applications in control and engineering, such as wind tunnels.

Alerted readers wonder that how to understand the term $p(\rho)=\rho^\gamma$ for a delta shock solution to \eqref{1.1}, in which $\rho$ contains a Dirac measure. This term does not exit for the case of pressureless flows, or the Chaplygin gas. For the latter, it is natural to take $p(\rho)=-1/\rho$ to be zero on the support of Dirac measures contained in $\rho$. So what about the case $\gamma>1$? This is a long standing obstruction to accept delta shock solutions for general Euler equations. However, to our experience, this is a spurious concern. The previous studies of physical problems demonstrated  that the state equation $p(\rho)=\rho^\gamma$ is useless when $\rho$ concentrates; we can determine a unique delta shock with all the other conditions in the definition of Radon measure solutions. Intuitively speaking, concentration of mass for the isentropic Euler equations \eqref{1.1} is not a thermotic, but a kinetic phenomena. Therefore, defining pressure for delta shocks (concentrated mass, or particles) is meaningless for this simplified model of isentropic gas.

It is remarkable that by the definition of Radon measure solutions, similar analysis as that in this paper  can also be undertaken for the general case that on the left of the delta shock, the state equation of the gas is $p(\rho)=A_{1}\rho^{\gamma_{1}}$, and on the right, it is $p(\rho)=A_{2}\rho^{\gamma_{2}}$, with different constants $A_{1}, A_{2}$ and $\gamma_{1}, \gamma_{2}$. This is a special type of multi-phase flow, or, mathematically speaking, a hyperbolic system of conservation laws with discontinuous fluxes. The result will be reported in another work. Interested readers may also see \cite[Chapter 8]{holden} and \cite{gueshen} for theory of scalar conservation laws with discontinuous fluxes and \cite{lefloch2007} for applications. No delta shocks had been considered for such problems before.

At the end of this introduction, we describe briefly the structure of the paper.
In Section \ref{s2}, we firstly present a definition of Radon measure solutions to a linear scalar conservation law in a very natural way, and derive the generalized Rankine-Hugoniot conditions of delta shock solutions. Then the definition of Radon measure solutions to the compressible Euler system is followed, based on the further idea of nonlinear algebraic constraints on the Radon-Nikodym derivatives of certain measures. In Section \ref{sec3}, we characterize the subset of the $(u,\rho)$-phase plane in which the right state can be connected to the given left state by a single delta shock. Unlike standard shocks, such a set is not a curve, but regions bounded by two curves in the phase plane. The over-compressing entropy condition is checked for these delta shocks, as well as the impact on the profile of a delta shock by its initial velocity $u_0$ of concentrated mass. The last Section \ref{s3} is devoted to {the} solvability of Riemann problems with general initial states.  We construct solutions to \eqref{1.1}\eqref{1.01} in Section \ref{s3.1}, and compare them with the classical Riemann solutions with the same initial data. See Figure \ref{fig7} and Table \ref{tab1}.  In Section \ref{s3.2}, we construct Radon measure solutions to \eqref{1.1}\eqref{1.2} for $\rho_0>0$. The solutions  are in general no longer self-similar. There is mass concentration for at least a short time,  no matter what the left and right initial states are, which is different from the case $\rho_0=0$. The solutions depend also on the initial velocity $u_0$ of the concentrated mass. The main results of this paper is Theorems \ref{thm3.1} and \ref{thm4.1}.
The classical Riemann solutions to \eqref{1.1}\eqref{1.01} are listed in Appendix.

\section{Radon measure solutions to Euler equations}\label{s2}

In this section, after presenting a definition of Radon measure solutions to a scalar conservation law and the compressible Euler equations, we introduce the generalized Rankine-Hugoniot conditions and over-compressing entropy condition of delta shocks.

Let $\mathscr{M}(\mathbb{R})$ be the space of signed Radon measures on the real line $\mathbb{R}$. It is the dual space of $C_c(\mathbb{R})$ consisting of compactly supported real-valued continuous functions on $\mathbb{
R}$. We call $m: [0, +\infty)\rightarrow \mathscr{M}(\mathbb{
R})$ a weak continuous mapping, if for all $\phi(x)\in C_{c}(\mathbb{R})$, the function
\begin{equation}\label{1.3}
[0,+\infty)\to\mathbb{R}:
\qquad t\mapsto\langle m(t),\phi(x)\rangle\doteq\int_{\mathbb{
R}}{\phi(x)\mathrm{d}m(t)}
\end{equation}
is continuous. The collection of such mappings is denoted by  $C([0, +\infty); \mathscr{M}(\mathbb{
R}))$.
Naturally, the (generalized) derivative of a measure $m(t)\in\mathscr{M}(\mathbb{
R})$ with respect to the space variable $x$, denoted as $m(t)_{x}$, is a distribution on the space $C_c^1(\mathbb{
R})$, of continuously differentiable functions with compact supports in $\mathbb{
R}$, defined by
\begin{equation}\label{eq2}
\langle m(t)_{x},\phi(\cdot)\rangle\doteq-\langle m(t), \ \phi_{x}(\cdot)\rangle,~~~\forall  \,\phi\in C_{c}^{1}(\mathbb{
R}).
\end{equation}
The derivative of $m(t)$ with respect to time $t$, $(m(t))_t=m'(t)$, is defined through the following weak derivative of the real-valued function $t\mapsto\langle m(t),\phi(x)\rangle$:
\begin{equation}\label{eq1}
\langle m'(t),\phi(\cdot)\rangle=\frac{\dd}{\dd t}\langle m(t),\phi(\cdot)\rangle,~~~\forall\,  \phi\in C_{c}(\mathbb{
R}).
\end{equation}

An important example is the Dirac measure supported on a Lipschitz curve $L$, given by $x=x(t)$ for $t\in [0,T)$,
with a weight $\omega_L(t)\in C^1([0,T))$. It is denoted by $\omega_L(t)\delta_{\{x=x(t)\}}$, and
\begin{equation}\label{eq:quxiancedu}
\langle\omega_L(t)\delta_{\{x=x(t)\}},\phi(\cdot)\rangle\doteq\omega_L(t)\phi(x(t)),~~~\forall \, \phi(x)\in C_{c}(\Bbb R).
\end{equation}
One checks easily that $\big(\omega_L(t)\delta_{\{x=x(t)\}}\big)_x=\omega_L(t)\delta'_{\{x=x(t)\}}$, with $\delta'_{\{x=x(t)\}}$ the dipole at the point $x(t)$, and $$\big(\omega_L(t)\delta_{\{x=x(t)\}}\big)_t=\omega'_L(t)\delta_{\{x=x(t)\}}+\omega_L(t)x'(t)\delta'_{\{x=x(t)\}}.$$

\begin{definition}\label{def2.0}  We say $(m(t), n(t))\in (C([0,+\infty);\mathscr{M}(\Bbb R)))^2$ is a Radon measure solution to the scalar conservation law
\begin{equation}\label{2.28}
m_{t}+n_{x}=0, \qquad t>0, \ x\in \Bbb R,
\end{equation}
with initial datum $m_0$, if $m(0)=m_0$, and \eqref{2.28} holds in the sense specified above, i.e., for any $\phi(x), \psi(t)\in C^1_c(\mathbb{R})$:
\begin{equation}\label{2.308new}
  \int_{0}^{+\infty}\langle m(t),\phi(\cdot)\rangle\psi'(t)\,\dd t+\int_{0}^{+\infty}\langle n(t),\phi'(\cdot)\rangle\psi(t)\,\dd t+\psi(0)\langle m_0, \phi(\cdot)\rangle=0.
\end{equation}\qed
\end{definition}

By an application of Stone-Weierstrass theorem, it is well-known that linear combinations of functions of the form $\phi\otimes\psi(x,t)\doteq \phi(x)\psi(t)$ are dense in $C_c^1(\mathbb{R}^2)$, the set of continuously differentiable functions with compact supports on the $(x,t)$-plane \cite[Chaper 4, Theorem 3]{Schwartz}. Thus
we infer that \eqref{2.308new} could be  written equivalently as
\begin{equation}\label{2.308}
  \int_{0}^{+\infty}\langle m(t),\varphi_{t}(\cdot,t)\rangle\dd t+\int_{0}^{+\infty}\langle n(t),\varphi_{x}(\cdot,t)\rangle\dd t+\langle m(0), \varphi(\cdot,0)\rangle=0,\ \forall  \varphi\in C_c^1(\mathbb{R}^2).
\end{equation}
This enables us to derive the  following generalized  Rankine-Hugoniot conditions for delta shocks:

\begin{lemma}\label{lm2.1}
Assume that $x=x(t)$ is a Lipschitz curve with $x(0)=0$, and
\begin{equation}\label{2.34}
m(t)=m_{l}(x,t)\mathcal{L}^{1}\lfloor\{x<x(t)\}+m_{r}(x,t)\mathcal{L}^{1}\lfloor\{x>x(t)\}+w_{m}(t)\delta_{\{x=x(t)\}},
\end{equation}
\begin{equation}\label{2.35}
n(t)=n_{l}(x,t)\mathcal{L}^{1}\lfloor\{x<x(t)\}+n_{r}(x,t)\mathcal{L}^{1}\lfloor\{x>x(t)\}+w_{n}(t)\delta_{\{x=x(t)\}}
\end{equation}
are piecewise smooth Radon measures solving the scalar conservation law
\begin{equation}\label{2.36}
m'(t)+n(t)_{x}=0, \qquad \ t>0, \ x\in \Bbb R
\end{equation}
with {initial data}
\begin{equation}\label{2.371}
\left \{
\begin{split}
& m(0)=m_{l}(x,0)\mathcal{L}^{1}\lfloor\{x<0)\}+m_{r}(x,0)\mathcal{L}^{1}\lfloor\{x>0\} +w_{m}^0\delta_{\{x=0\}},\\
&n(0)=n_{l}(x,0)\mathcal{L}^{1}\lfloor\{x<0)\}+n_{r}(x,0)\mathcal{L}^{1}\lfloor\{x>0\}+w_{n}^0\delta_{\{x=0\}},
\end{split} \right.
\end{equation}
where $m_{l}, m_{r},  n_{l}, n_r$ are all continuously differentiable functions satisfying the equation \eqref{2.36} point-wisely in their respective domain of definition, and continuous up to the boundary for $t>0$; moreover, $w_m^0, w_n^0$ are constants.
Then it holds along $x=x(t)$ that
 \begin{equation}\label{2.37}
\left \{
\begin{split}
&\frac{\mathrm{d}(w_{m}(t))}{\mathrm{d}t}=[m]x'(t)-[n],\\
&w_{m}(0)=w_m^0, \ \  w_{n}(t)=w_{m}(t)x'(t),
\end{split} \right.
\end{equation}
with $$[m]\doteq m_{r}(x(t)+, t)-m_{l}(x(t)-, t), \ [n]\doteq n_{r}(x(t)+, t)-n_{l}(x(t)-, t).$$
\end{lemma}

The proof of this lemma depends on straightforward integration-by-parts and the arbitrariness of test functions $\varphi$ in \eqref{2.308}. Details can be found in \cite[Lemma 2.1, p.2672]{Jin2021Radon}.


\begin{definition}\label{def2.1}
Let $\varrho\in C([0,+\infty);\mathscr{M}(\mathbb{R}))$, and $u(t)$ be a $\varrho(t)$-measurable function on $\mathbb{R}$ for $t\ge0$. 
We call $(\varrho(t),\ u(t))$ a {\it Radon measure solution} to the Riemann problem \eqref{1.1}\eqref{1.2}, 
if there exist $m,\ n,\ \wp\in C([0,+\infty);\mathscr{M}(\mathbb{R}))$ satisfying
\par i) [linear relaxation] for any $\varphi\in C_{c}^{1}(\Bbb R^2),$  there hold
\begin{align}\label{2.3}
\int_{0}^{+\infty}\langle \varrho(t),\varphi_t(\cdot,t)\rangle\, \dd t+\int_{0}^{+\infty}\langle m(t),\varphi_x(\cdot,t)\rangle \,\dd t+\langle \varrho(0),\varphi(\cdot,0)\rangle =0,
\end{align}
\begin{equation}\label{2.4}
\begin{split}
\int_{0}^{+\infty}\langle m(t),\varphi_t(\cdot,t)\rangle \,\dd t&+\int_{0}^{+\infty}\langle n(t),\varphi_x(\cdot,t)\rangle \,\dd t
\\&+\int_{0}^{+\infty}\langle \wp(t),\varphi_x(\cdot,t)\rangle\, \dd t+\langle m(0),\varphi(\cdot,0)\rangle=0;
\end{split}\end{equation}

\par ii) [nonlinear constraints] $\varrho(t)$ is  non-negative for all $t\ge0$, such that $(m(t), \ n(t))$ is absolutely continuous with respect to $\varrho(t)$ (written as $(m, n)\ll\varrho$),
and the corresponding Radon-Nikodym derivatives satisfy for any $t\ge0$ and $\varrho(t)$-a.e. $x\in\mathbb{R}$ that
\begin{equation}\label{2.5}
\dfrac{\mathrm{d} m(t)(x) }{\mathrm{d} \varrho(t)(x)}\doteq u(x,t), \ \ \  u(x,t)^{2}=\dfrac{\mathrm{d} n(t)(x) }{\mathrm{d} \varrho(t)(x)};
\end{equation}
furthermore, $\wp\ll \varrho$ on the set where $\varrho\ll\mathcal{L}^1$; and for this case, there holds the state equation
\begin{equation}\label{2.6}
p(x,t)=p(\rho(x,t)), \qquad\mathcal{L}^1\text{-a.e.}~x,
\end{equation}
where $\rho(x,t)=\dd\varrho(t)/\dd\mathcal{L}^1$, and $p(x,t)=\dd\wp(t)/\dd\mathcal{L}^1$;

\par iii) [entropy condition] on the set where $\varrho\ll\mathcal{L}^1$,  all discontinuities in the flow field $(\rho, u)$ satisfy the classical Lax condition.\qed
\end{definition}
\begin{remark}\label{rem21}
It is easy to see that integral weak solutions could be identified as Radon measure solutions.  It is also natural that on the complement of support of $\varrho(t)$, the velocity $u(x,t)$ is meaningless, hence not defined. However, by \eqref{2.5}, it is still possible that $u$ is defined and nonzero on a $\varrho(t)$-measure null set.\qed
\end{remark}
\begin{remark}\label{rem22new}
We recognized that the idea of considering mass and momentum as measures had appeared in \cite{E1996Generalized} and \cite{lizhangyang}. Then the velocity is derived from the Radon-Nikodym derivatives, as shown in the above definition. It seems that this approach is more fundamental, as it could be easily generalized to study multidimensional problems.
\end{remark}
\begin{definition}\label{def22} Let $0<T\le+\infty$ and $t\in[0,T]$.
A piecewise-smooth Radon measure solution $(\varrho(t), u(x,t))$ to \eqref{1.1}\eqref{1.01} is called a {\it delta shock solution}, if:~1)~$\varrho(t)$ contains a Dirac measure supported  
at $\{x(t)\}$, where $x=x(t)$ is a Lipschitz curve, and  $T\ge t\ge0$; ~2)~for any $T\ge t\ge0$, on the region $\Omega_{1}\doteq\{x\in\mathbb{R}:~x<x(t)\}$, $\Omega_{2}\doteq\{x\in\mathbb{R}:~ x>x(t)\}$, $\varrho(t)$ is absolutely continuous with respect to the Lebesgue measure $\mathcal{L}^1$,  with 
the Radon-Nikodym derivative $\rho(x,t)$;  ~3)~both $u(x,t)$ and $\rho(x,t)$, as functions of $(x,t)$, are differentiable in $\{(x,t): x<x(t), T\ge t\ge0\}\cup\{(x,t): x>x(t), T\ge t\ge0\}$, and continuous up to the boundary $\{(x(t),t): T\ge t\ge0\}$. The curve $\{x=x(t), 0\le t\le T\}$ is called a delta shock-front, or simply  a delta shock.\qed
\end{definition}

A delta shock solution is called global, if $T=+\infty$; and local, if $T<+\infty$.

We consider the following admissibility condition of delta shocks. It resembles the entropy conditions for \eqref{eq13} conjectured by Tan, Zhang, and Keyfitz,  Kranzer, mentioned in the articles cited in Introduction. It is reasonable from the view point of sticky particles. However, its relevance to a well-posedness theory on the Radon measure solutions to general compressible Euler equations is still unclear.
\begin{definition}\label{def23}
We say a delta shock solution $(\varrho,u)$ to \eqref{1.1}\eqref{1.2} satisfying the over-compressing entropy condition, if
\begin{equation}\label{2.66}
u_{l}(t)\geq x'(t)\geq u_{r}(t),
\end{equation}
with $u_l(t)\doteq u(x(t)-,t)$ and $u_r(t)\doteq u(x(t)+,t)$.\qed
\end{definition}

\section{Delta shocks and their properties}\label{sec3}
In this section, we focus on the existence, entropy conditions, and  geometric properties of delta shock solutions to the Euler system \eqref{1.1}.

Firstly, we specify the right constant state $U_2=(u_2, \rho_2)$ that can be connected to the given left constant state $U_1=(u_1, \rho_1)$ by a delta shock, which means we consider a solution to
\eqref{1.1}\eqref{1.2} of the form
\begin{align}\label{2.7}
&\varrho(t) \doteq\rho_{1}\mathcal{L}^{1}\lfloor\Omega_{1}+\rho_{2}\mathcal{L}^{1}\lfloor\Omega_{2} +w_{\rho}(t)\delta_{\{x=x(t)\}},\\ \label{3.2}
&m(t)\doteq\rho_{1}u_{1}\mathcal{L}^{1}\lfloor\Omega_{1}+\rho_{2}u_{2}\mathcal{L}^{1}\lfloor\Omega_{2} +w_{m}(t)\delta_{\{x=x(t)\}},
\\ \label{2.8}
&n(t)\doteq\rho_{1}u_{1}^{2}\mathcal{L}^{1}\lfloor\Omega_{1} +\rho_{2}u_{2}^{2}\mathcal{L}^{1}\lfloor\Omega_{2}+w_{n}(t)\delta_{\{x=x(t)\}},
\\  \label{2.9}
&\wp(t)\doteq p_{1}\mathcal{L}^{1}\lfloor\Omega_{1} +p_{2}\mathcal{L}^{1}\lfloor\Omega_{2}+w_{p}(t)\delta_{\{x=x(t)\}}.
\end{align}
Here, $\Omega_{1}$ and $\Omega_{2}$ were defined in Definition \ref{def22}, with $x=x(t)$ the delta shock front to be determined, which satisfies $x(0)=0$. The weights,  $w_{\rho}(t),\ w_{m}(t), \ w_{n}(t),$ and $w_{p}(t)$,  are all functions of $t$, to be solved. The non-negativeness of $\varrho$  requires that $w_\rho(t)\geq0$.

Applying  Lemma \ref{lm2.1} to \eqref{2.3} (which is $\varrho_t+m_x=0$), with
initial data
\begin{equation}
\left \{
\begin{split}
&\varrho_0=\rho_{1}\mathcal{L}^{1}\lfloor\{x<0\}+\rho_{0}\delta_{\{x=0\}}+\rho_{2}\mathcal{L}^{1}\lfloor\{x>0\},\\
&m_0=\rho_1u_{1}\mathcal{L}^{1}\lfloor\{x<0\}+\rho_0u_{0}\delta_{\{x=0\}}+\rho_2u_{2}\mathcal{L}^{1}\lfloor\{x>0\} \\
\end{split} \right.
\end{equation} provided by \eqref{1.2}, we have
\begin{align}\label{2.10}
&w_{\rho}(0)=\rho_{0}, \ \  \  w_{m}(t)=w_{\rho}(t)x'(t),
\\ \label{2.11}
&\frac{\mathrm{d} (w_{\rho}(t))}{\mathrm{d} t}=[\rho]x'(t)-[\rho u],
\end{align}
where $[\rho]\doteq \rho_{2}-\rho_{1},~ [\rho u]\doteq \rho_{2}u_{2}-\rho_{1}u_{1}. $ It follows that
\begin{equation}\label{2.12}
w_{\rho}(t)=[\rho]x(t)-[\rho u]t+\rho_{0}.
\end{equation}
In a similar way, applying the generalized Rankine-Hugoniot conditions of \eqref{2.4} (which is $m_t+(n+\wp)_x=0$), one has
\begin{align}\label{2.13}
&w_{m}(0)=\rho_{0}u_{0}, \ \  \   w_{n}(t)+w_{p}(t)=w_{m}(t)x'(t),
\\ \label{2.14}
&\frac{\mathrm{d} (w_{m}(t))}{\mathrm{d} t}=[\rho u]x'(t)-[\rho u^{2}+p],
\end{align}
where $[\rho u^{2}+p]\doteq \rho_{2}u_{2}^{2}-\rho_{1}u_{1}^{2}+ p_{2}-p_{1}.$ Consequently,
\begin{equation}\label{2.15}
w_{m}(t)=[\rho u]x(t)-[\rho u^{2}+p]t+\rho_{0}u_{0}.
\end{equation}
By  \eqref{2.5}, we have $w_{n}(t)=w_{m}(t)x'(t)$. Then $\eqref{2.13}_2$ implies that $w_{p}(t)\equiv0.$

Recalling \eqref{2.5} and \eqref{2.7}, once we know the delta shock front $x=x(t)$, we obtain the Radon measure solution to problem \eqref{1.1}\eqref{1.2} that contains only a delta shock:
\begin{align}\label{2.151}
&\varrho(t)=\rho_{1}\mathcal{L}^{1}\lfloor\Omega_{1}+\rho_{2}\mathcal{L}^{1}\lfloor\Omega_{2}
+([\rho]x(t)-[\rho u]t+\rho_{0})\delta_{\{x=x(t)\}},\\
\label{2.152}
&u(t)=u_{1}\mathsf{I}_{\Omega_{1}}+u_{2}\mathsf{I}_{\Omega_{2}}+x'(t)\mathsf{I}_{\{x=x(t)\}}.
\end{align}

Thanks to $w_{m}(t)=w_{\rho}(t)x'(t)$ from \eqref{2.10},  equations \eqref{2.12} and \eqref{2.15} yield  the following Cauchy problem for a nonlinear ordinary differential equation of $x=x(t)$:
{\begin{equation}\label{2.16}
\left \{\begin{split}
&[\rho u]x(t)-[\rho u^{2}+p]t+\rho_{0}u_{0}=\big([\rho]x(t)-[\rho u]t+\rho_{0}\big)x'(t), \\&x(0)=0.
\end{split}\right.
\end{equation}}
Integrating \eqref{2.16} from 0 to $t$ gives
\begin{equation}\label{2.17}
\frac{[\rho]}{2}x^{2}(t)-([\rho u]t-\rho_{0})x(t)
=-\frac{[\rho u^{2}+p]t^{2}}{2}+\rho_{0}u_{0}t.
\end{equation}
Set
\begin{align}\label{eq39}
   & \Delta\doteq at^{2}+2\rho_{0}bt+\rho_{0}^{2}, \ \  b\doteq [\rho]u_{0}-[\rho u],\\
   & a\doteq [\rho u]^{2}-[\rho][\rho u^{2}+p]=\rho_{1}\rho_{2}[u]^{2}-[\rho][p].\label{eq40}
\end{align}
 We have the following conclusions about the existence of $x(t)$:
\begin{lemma}\label{lem2.3}
Consider the Riemann problem \eqref{1.1}\eqref{1.2} for polytropic gases with a solution given by \eqref{2.151}-\eqref{2.152}.

\begin{itemize}
\item[($\spadesuit$)] For $\rho_0=0$, the existence of solution  requires that  $a\geq 0$. More precisely,
  \begin{itemize}
    \item[1)]  there is a unique global delta shock solution with $x(t)=\frac{(u_{1}+u_{2})t}{2}$ and $w_{\rho}(t)=-\rho_{1}[u]t$ if $[\rho]=0$ and $[u]\le0$;
    \item[2)]  there is a  unique global delta shock solution with $x(t)=\frac{([\rho u]+\sqrt{a})t}{[\rho]}$ and $w_{\rho}(t)=\sqrt{a}\,t$ if $[\rho]\neq 0$ and $a\geq 0$.
  \end{itemize}
\item[($\spadesuit$)]  For $\rho_0>0$,  the existence of solution  requires that  $\Delta\geq0$. More precisely,
  \begin{itemize}
    \item[3)] if $[\rho]=0$, then  there is a unique global delta shock solution with  $x(t)=\frac{\rho_1 [u^{2}]t^{2}-2\rho_{0}u_{0}t}{2(\rho_1 [u]t-\rho_{0})}$ and $w_{\rho}(t)=-\rho_{1} [u]t+\rho_{0}$ if $[u]\leq 0$; and a unique local delta shock  solution for $t\in[0,t_*)$ with $t_*=\frac{\rho_0}{\rho_{1}[u]}$ if $[u]>0$. Moreover, $w_\rho(t_*-)=0$ and $x(t_*)=+\infty$ ($-\infty$, respectively) if $u_0\geq\frac{u_1+u_2}{2}$ ($u_0<\frac{u_1+u_2}{2}$, respectively);

    \item[4)] if $[\rho]\neq 0$, then there is a unique  global delta shock solution with  $x(t)=\frac{[\rho u]t-\rho_{0}+\sqrt{\Delta}}{[\rho]}$ and  $w_{\rho}(t)=\sqrt{\Delta}$,  if the initial data hold $a\geq 0$ and $ b\geq -\sqrt{a}$; and a  unique local delta shock solution for $t\in[0, t^{*}]$, if the initial data satisfy one of the following: ~i)~ $b<-\sqrt{a}$ and $a>0$; ~ii)~ $a<0$; ~iii)~ $a=0$ and $b<0$. Here $t^{*}=\frac{-\rho_{0} b-\rho_{0}\sqrt{b^{2}-a}}{a}$ for $a\ne0$ and $t^*=-\frac{\rho_0}{2b}$ for $a=0$. Moreover, $x(t^*)$ is { finite, 
        $w_\rho(t^*-)=0$ and $x'(t^*-)=\infty$.}
  \end{itemize}
  \end{itemize}
\end{lemma}
\begin{remark}\label{rmk2new}
  The solutions to the (singular) Riemann problems for pressureless Euler equations could be found in \cite[Section 3.4]{lizhangyang}, see also \cite{Gao2021Free}, which corresponds to the case $[p]\equiv0$ in \eqref{eq40} and is much simpler. The differences reflect the role played by the pressure.\qed
\end{remark}

\begin{proof}
By the preceding analysis, we only need to focus on the solvability of $x(t)$ satisfying \eqref{2.17}.

1. For $[\rho]=0,$  we have $[p]=0,$ and \eqref{2.17} is reduced to
\begin{equation}\label{2.18}
2(\rho_{1} [u]t-\rho_{0})x(t)=\rho_1 [u^{2}]t^{2}-2\rho_{0}u_{0}t.
\end{equation}
It follows that \begin{align}\label{eq44}
  x(t)=\frac{\rho_{1} [u^{2}]t^{2}-2\rho_{0}u_{0}t}{2(\rho_{1} [u]t-\rho_{0})}.
\end{align}
By \eqref{2.12}, we find
\begin{equation}\label{2.18a}
  w_{\rho}(t)=-\rho_{1} [u]t+\rho_{0}.
\end{equation}

If $\rho_0=0$, then $x(t)=\frac{(u_{1}+u_{2})t}{2}$ and $w_{\rho}(t)=-\rho_{1}[u]t$. For this case, the non-negativity of $w_\rho(t)$ is equivalent to $[u]\leq 0,$ while for $[u]=0$ there is no discontinuity.  We complete the proof of item 1).

If $\rho_0>0$, it is easy to see that if $[u]\le0$, then $w_{\rho}(t)>0$ for $t>0$. However, if $[u]>0$,  $x(t)$ blows up as $t\nearrow t_{*}=\frac{\rho_{0}}{\rho_{1} [u]},$ and $ w_{\rho}(t_{*})=0.$ Moreover,
\begin{align*}
\lim_{t\rightarrow t_{*}-}x(t)&=\lim_{t\rightarrow t_{*}-}\frac{\rho_{1} [u^{2}](t^{2}-(t_{*}^{-})^{2})+\rho_{0}(u_1+u_2-2u_{0})\frac{\rho_{0}}{\rho_{1} [u]}}{2(\rho_{1}[u]t-\rho_{0})}\\
&=
\begin{cases}
+\infty, \ \ \ &{\rm if~}u_{0}\geq\frac{u_{1}+u_{2}}{2},\\
-\infty, \ \ \ &{\rm if~}u_{0}<\frac{u_{1}+u_{2}}{2}.
\end{cases}
\end{align*}
This completes the proof of item 3).

2. For $[\rho]\neq 0$, it follows from \eqref{2.17} that
\begin{equation}\label{2.228}
x(t)=\frac{[\rho u]t-\rho_{0}\pm \sqrt{\Delta}}{[\rho]}.
\end{equation}
Combining it with  {\red \eqref{2.12}} gives $w_{\rho}(t)=
\pm \sqrt{\Delta}$. Nonnegativity of $w_{\rho}(t)$ requires that
\begin{equation}\label{wrho}
 w_{\rho}(t)=\sqrt{\Delta}.
\end{equation}
Correspondingly the delta shock front  is
\begin{equation}\label{2.228a}
x(t)=\frac{[\rho u]t-\rho_{0}+ \sqrt{\Delta}}{[\rho]}.
\end{equation}

If $\rho_0=0$, $\ a\geq 0$, then $\Delta=\sqrt{a}t\geq0$. Thus there is a global delta shock front $x(t)=\frac{([\rho u]+\sqrt{a})t}{[\rho]}$, with $w_{\rho}(t)=\sqrt{a}t$. We complete the proof of item 2).

For the case $\rho_{0}>0,$ if $a\geq 0,\ b\geq -\sqrt{a}$, then $\Delta=at^{2}+2\rho_{0}bt+\rho_{0}^{2}\geq 0$ for all $t\geq 0$.  It is easy to check that  $\Delta\geq0$ for $t\in[0,t^*]$ with some positive $t^*$ if there holds one of the three cases: i)  $a>0, \ b< -\sqrt{a}$; ii) $a<0;$ iii) $a=0$, $b<0$. 
In particular, 
$$t^{*}=\begin{cases}\displaystyle\frac{-\rho_{0} b-\rho_{0}\sqrt{b^{2}-a}}{a}, &\mbox{if}~a\ne0,\\
\displaystyle-\frac{\rho_0}{2b},&\mbox{if}~a=0.
\end{cases}$$
  Moreover,  $x(t^*-)=\frac{[\rho u]t^*-\rho_{0}}{[\rho]}$ and $w_\rho(t^*-)=\sqrt{\Delta}=0$. It follows from \eqref{2.228a} that $$x'(t)=\frac{[\rho u]\sqrt{\Delta}+at+\rho_{0} b}{[\rho]\sqrt{\Delta}},$$ hence $$x'(t^*-)=\lim_{t\rightarrow t^*-}
\frac{-\rho_{0}\sqrt{b^{2}-a}}{[\rho]\sqrt{\Delta}}=
\infty.$$ 
The proof of item 4) is completed. \qed
\end{proof}

The results of Lemma \ref{lem2.3} are summarized in Table \ref{tab2}, except for the case $\rho_0=0$, $[\rho]=0,$ $[u]=0$, for which the solution  is a trivial constant state.
 \begin{table}[t]
\centering{\small
\caption{Existence of a single delta shock solution. ($a,b$ and $\Delta$ were defined by \eqref{eq39}\eqref{eq40}.)}\label{tab2}
 \resizebox{\textwidth}{40mm}{
 \begin{tabular}{|p{2.5em}|p{3.1em}|p{10em}|p{7.5em}|p{3.5em}|p{7.5em}|p{3em}|}
\hline
\multirow{3}{*}{$\rho_0=0$}&{$[\rho]=0$}&$x(t)=\frac{(u_{1}+u_{2})t}{2}$ &{$w_{\rho}(t)=-\rho_{1}[u]t$} &{$[u]<0$} &\multirow{2}{*}{$w_{\rho}(t)\nearrow \infty$} &\multirow{5}{*}{\makecell{global \\ delta \\shock\\ solution}}\\
 \cline{2-5}~&\multirow{2}{*}{$[\rho]\neq 0$}&\multirow{2}{*}{$x(t)=\frac{([\rho u]+\sqrt{a})t}{[\rho]}$} &\multirow{2}{*}{$w_{\rho}(t)=\sqrt{a}t$}&{$a>0$}&~&~\\
  \cline{5-6}~&~&~&~&{$a=0$}&{$w_{\rho}(t)=0$}&~\\
\cline{1-6}%
\multirow{22}{*}{$\rho_0>0$}& \multirow{3}{*}{$[\rho]=0$}& \multirow{3}{*}{$x(t)=\frac{\rho_1[u^{2}]t^{2}-2\rho_{0}u_{0}t}{2(\rho_1 [u]t-\rho_{0})}$}& \multirow{3}{*}{$\makecell{w_{\rho}(t)=-\rho_{1}[u]t\\ \,\,\,\,+\rho_0}$} &{$[u]<0$}&{$w_{\rho}(t) \nearrow \infty$}&~\\
  \cline{5-6}~&~&~&~&{$[u]=0$}&{$w_{\rho}(t)=\rho_0$}&~\\
\cline{5-7}~&~&~&~&{$[u]>0$}&\multirow{9}{*}{$w_{\rho}(t)\searrow 0$}&\multirow{9}{*}{\makecell{local \\ delta \\shock \\ solution}}\\
 \cline{2-5}
 ~&~&~&~&{\makecell{$a>0,$ \\ $b<-\sqrt{a}$}}&~&~\\
 \cline{5-5}~&~&~&~&{\makecell{$a=0,$ \\ $ b<0$}}&~&~\\
 \cline{5-5}~&~&~&~&{\makecell{$a<0,$ \\ $ b<0$}}&~&~\\
 \cline{5-6}~&~&~&~&{\makecell{$a<0, $ \\ $b\geq 0$}}&{$w_{\rho}(t)\nearrow\searrow 0$}&~ \\
 \cline{5-7}
 ~&{$[\rho]\neq 0$}&{ $x(t)=\frac{[\rho u]t-\rho_{0}+\sqrt{\Delta}}{[\rho]}$ }&{$w_{\rho}(t)=\sqrt{\Delta}$}
 &{\makecell{$a>0,$ \\ $ b>-\sqrt{a}$}}&\multirow{3}{*}{$w_{\rho}(t)\nearrow \infty$} &\multirow{8}{*}{\makecell{global \\ delta \\ shock\\ solution}}\\
  \cline{5-6}~&~&~&~&{\makecell{$a>0, $ \\ $ b=-\sqrt{a}$}}&\multirow{3}{*}{$w_{\rho}(t)\searrow 0 \nearrow\infty$}&~ \\
 \cline{5-6}~&~&~&~&{\makecell{$a=0,$ \\ $b>0$}}&{$w_{\rho}(t)\nearrow \infty$}&~ \\
 \cline{5-6}~&~&~&~&{\makecell{$a=0, $ \\ $ b=0$}}&{$w_{\rho}(t)=\rho_0$}&~ \\
\hline
\end{tabular}}}
\end{table}

We now illustrate the set of initial data in the upper-half $(u,\rho)$-plane that admit a single delta shock for $\rho_0=0$.
From the equation $a\doteq\rho_{1}\rho_2[u]^{2}-[\rho][p]=0$, we get four curves passing $(u_1, \rho_1)$ as follows:
\begin{equation*}
S_{1}\doteq\left\{(u,\rho)~|~ u-u_{1}=-\sqrt{\frac{(\rho-\rho_{1})(p-p_{1})}{\rho \rho_{1}}}, \ \  \rho>\rho_{1}\right\},
\end{equation*}
\begin{equation}\label{2.47} S_{11}\doteq\left\{(u,\rho)~|~ u-u_{1}=\sqrt{\frac{(\rho-\rho_{1})(p-p_{1})}{\rho \rho_{1}}},  \ \ \rho<\rho_{1}\right\},
\end{equation}
\begin{equation*}S_{2}\doteq\left\{(u,\rho)~|~ u-u_{1}=-\sqrt{\frac{(\rho-\rho_{1})(p-p_{1})}{\rho \rho_{1}}}, \ \ \rho<\rho_{1}\right\},
\end{equation*}
\begin{equation}\label{2.48}S_{22}\doteq\left\{(u,\rho)~|~  u-u_{1}=\sqrt{\frac{(\rho-\rho_{1})(p-p_{1})}{\rho \rho_{1}}}, \ \ \rho>\rho_{1}\right\}.
\end{equation}
Note that these curves are exactly the Rankine-Hugoniot loci derived from the classical Rankine-Hugoniot conditions of shock-fronts.

Let $R_1$ and $R_2$ be the rarefaction wave curves as indicated in Appendix (see \eqref{3.17} and \eqref{3.18}). Then they, together with $S_1$ and $S_2$, divide the upper half-plane $\{(u,\rho):\rho>0\}$ into five regions $\mathrm{I}\cup \mathrm{II}\cup \mathrm{III}\cup \mathrm{IV}\cup \mathrm{V}$, see Figure \ref{fig8}.
\begin{figure}[t]
\centering
  \includegraphics[width=0.6\textwidth]{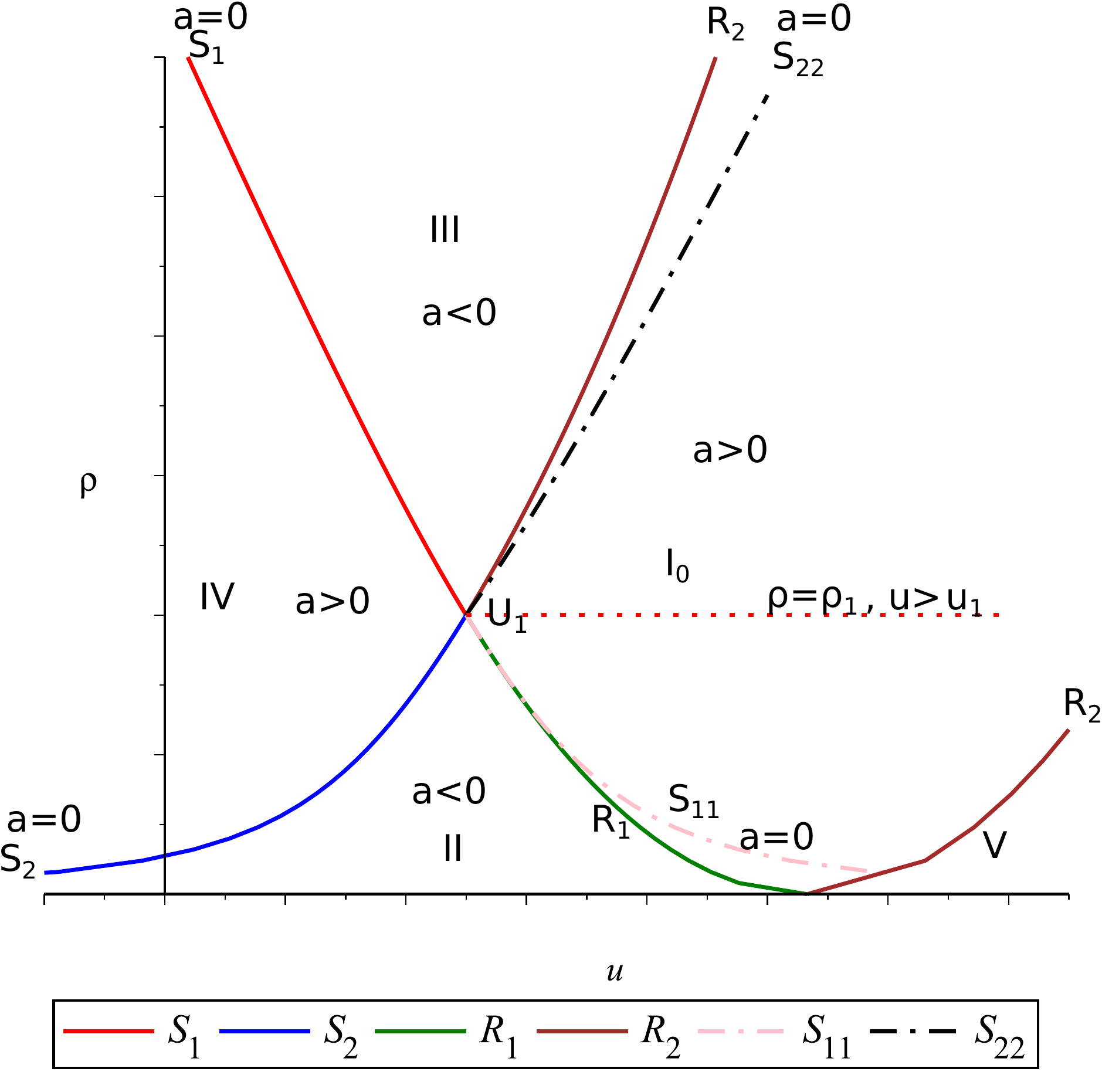}
\caption{
For ${\rho_0=0}$ and given $U_1=(u_1,\rho_1)$, those $U_2$ lying in $\big(\mathrm{I_{0}}\setminus\big\{(u,\rho)~|~\rho=\rho_1, u>u_1\big\}\big)\cup \mathrm{IV}$ could be connected to $U_1$ by a single  delta shock.}\label{fig22}
\end{figure}
We now prove that for any $(u,\rho)\in R_{2}\setminus \{U_{1}\},$ it holds $a(u,\rho)<0,$ which means that $R_2$ is on the left-hand-side of $S_{22},$ as drawn in Figure \ref{fig22}.
By \eqref{3.18} and \eqref{2.48}, we need to show that
\begin{equation}\label{eq49new}
\begin{split}H(\rho)&\doteq\frac{2\sqrt{\gamma}}{\gamma-1}
\Big(\rho^{\frac{\gamma-1}{2}}-\rho_1^{\frac{\gamma-1}{2}}\Big)-\sqrt{\frac{(\rho-\rho_{1})(p-p_{1})}{\rho \rho_{1}}}
\\&=\frac{\frac{4\gamma}{(\gamma-1)^{2}}
\Big(\rho^{\frac{\gamma-1}{2}}-\rho_1^{\frac{\gamma-1}{2}}\Big)^{2}-\frac{(\rho-\rho_{1})(p-p_{1})}{\rho \rho_{1}}}{\frac{2\sqrt{\gamma}}{\gamma-1}
\Big(\rho^{\frac{\gamma-1}{2}}-\rho_1^{\frac{\gamma-1}{2}}\Big)+\sqrt{\frac{(\rho-\rho_{1})(p-p_{1})}{\rho \rho_{1}}}}<0, \qquad\forall\, \rho>\rho_1.
 \end{split}
\end{equation}
It is equivalent to prove that the numerator is negative, namely
\begin{equation*}
h(\rho)\doteq\frac{4\gamma \rho \rho_{1}}{(\gamma-1)^{2}}
\Big(\rho^{\frac{\gamma-1}{2}}-\rho_1^{\frac{\gamma-1}{2}}\Big)^{2}-(\rho-\rho_{1})(p-p_{1})<0, \qquad \forall\, \rho>\rho_1.
\end{equation*}
Some computation yields $h(\rho)=\frac{\rho_1^{\gamma+1}}{(\gamma-1)^{2}}G(A),$
where $A=\frac{\rho}{\rho_1}>1$, and
\begin{equation*}
\begin{split} G(A)&\doteq(\gamma+1)^{2}A(A^{\frac{\gamma-1}{2}}-1)^{2}-(\gamma-1)^{2}(A^{\frac{\gamma+1}{2}}-1)^{2}\\
 &=\underbrace{\{(\gamma+1)A^{\frac{1}{2}}(A^{\frac{\gamma-1}{2}}-1)+(\gamma-1)(A^{\frac{\gamma+1}{2}}-1)\}}_{>0}
 \\
 &\qquad\quad\cdot\underbrace{\{(\gamma+1)A^{\frac{1}{2}}(A^{\frac{\gamma-1}{2}}-1)-(\gamma-1)(A^{\frac{\gamma+1}{2}}-1)\}}_{\doteq g(A)}.
 \end{split}
\end{equation*}
Notice that
$g'(A)=\frac{\gamma+1}{2}A^{-\frac{1}{2}}g_1(A),$ with $g_{1}(A)\doteq\gamma A^{\frac{\gamma-1}{2}}-1-(\gamma-1)A^{\frac{\gamma}{2}}.$ Obviously  $g_{1}(1)=0,$ and
 $$g'_{1}(A)=\frac{\gamma(\gamma-1)}{2}(A^{\frac{\gamma-3}{2}}-A^{\frac{\gamma-2}{2}})<0 \ \text{for} \ A>1,\ \gamma>1.$$
Thus $g_{1}(A)<0$, $g'(A)<0,$ $g(A)<0,$ $G(A)<0,$ $h(\rho)<0$ and $H(\rho)<0.$

Similarly, it can be shown that $R_1$ is on the left-hand-side of $S_{11}$, as depicted in Figure \ref{fig22}. It is easy to check that $a>0$ in the domain $\mathrm{IV}$ on the left-hand-side of $U_1$ bounded by $S_1$ and $S_2$, and the domain $\mathrm{I}_0$ on the right-hand-side of $U_1$ bounded by $S_{11}$ and $S_{22}$, while $a<0$ in the other domains, see Figure \ref{fig22}.  Then the existence region of delta shock solution is $\big(\mathrm{I_{0}}\setminus\big\{(u,\rho)~|~\rho=\rho_1, u>u_1\big\}\big)\cup \mathrm{IV}$.
We remark that the solutions do not satisfy the continuous dependence on initial value for those initial data near the half-line $\{(u,\rho)~|~\rho=\rho_1,\ u>u_1\}.$

\begin{remark}\label{rmk4new}
The connection of a delta shock and a classical free piston could be explained by using the generalized Rankine-Hugoniot conditions. Here by a classical free piston, we mean a finite solid body (modeled by an intervals with fixed length) carrying mass $\rho_0>0$ in a tube (i.e., the $x$-axis), driving by the differences of gas pressures on both sides of it \cite{Gao2021Free}, and moving according to the Newton's second law; while for all the time the piston does not absorb or release any gas. Thus the gas moves with the same speed as that of the piston nearing its surface, and the mass of the piston does not change. A natural quotient topology of $\mathbb{R}$ suggests that the piston could be regarded as a point. Now from \eqref{2.11}, and the assumption that $[u]=0$ (the speed of gas is the same on two sides of the piston), as well as $w_\rho(t)\equiv \rho_0$ for all $t\ge0$, one has $x'(t)=u(x(t)-,t)=u(x(t)+,t)$, namely the piston has the same speed as the gas close to it. (If $[\rho]\equiv0$, then there is no discontinuity in the flow.) Now by \eqref{2.10}, the equation \eqref{2.14}  reads $\frac{\dd (\rho_0 x'(t))}{\dd t}=[p]$, which is exactly the Newton's second law. Of course, Lemma \ref{lem2.3} implies that such a classical piston could not be used to eliminate all thermal discontinuities (shocks) in the flow field, such that the flow becomes uniform.\qed
\end{remark}

\begin{remark}\label{rm31}
For cases in item 4),  since $w_\rho(t^*)=0$, we say that the delta shock disappears at the moment $t^*$, and the problem is reduced to the case of $\rho_0=0$ to prolong the solution,  provided that $x(t^*)$ is finite.\qed
\end{remark}

We now check that whether the delta shocks constructed above obey the over-compressing entropy condition \eqref{2.66}. 
Since $$b=[\rho]u_{0}-[\rho u]=\rho_{2}(u_{0}-u_{2})+\rho_{1}(u_{1}-u_{0}),$$ $b<0$ implies that at least one of $u_{0}-u_{2}$ and $u_{1}-u_{0}$ is less than zero. Hence the initial value does not satisfy \eqref{2.66} as long as $b<0$. Also \eqref{2.66} implies that the weight of mass $w_\rho(t)$ is a monotonically non-decreasing function of time, thus the global/local delta shock solutions obtained in Lemma \ref{lem2.3} satisfying \eqref{2.66} are only those with increasing or constant weight of mass. So we can exclude some cases that do not satisfy \eqref{2.66} from Table 1. Moreover, we have the following conclusions.
%

\begin{proposition}\label{prp1}
Consider the Riemann problem  \eqref{1.1}\eqref{1.2} for polytropic gases.
\begin{itemize}
\item[\textcircled{1}]~ If the initial data satisfy one of the following conditions:
\begin{itemize}
\item[\textcircled{a}] for $\rho_0=0$,
 { \begin{itemize}
  \item[{\rm i)}] $[\rho]=0$ and  $[u]<0$;
  \item[{\rm ii)}] $[\rho]>0, ~ [u]\leq 0, ~a\geq \rho_{1}^{2}[u]^{2}$;
  \item[{\rm iii)}]  $[\rho]<0, ~ [u]\leq 0,~ a\geq \rho_{2}^{2}[u]^{2}$;
 \end{itemize}}
\item[\textcircled{b}] for $\rho_0>0$,
  \begin{itemize}
    \item[{\rm iv)}] 
         $[\rho]=0,$ $ u_{2}\leq u_{0}\leq u_{1}$;
    \item[{\rm v)}] $[\rho]>0, ~ u_{2}\leq u_{0}\leq u_{1}, ~ a\geq \rho_{1}^{2}[u]^{2}$, $b\geq-\sqrt{a}$;
    \item[{\rm vi)}]$[\rho]<0, ~ u_{2}\leq u_{0}\leq u_{1},~ a\geq \rho_{2}^{2}[u]^{2}$, $b\geq-\sqrt{a}$,
  \end{itemize}
  \end{itemize}
      then there exists a unique global delta shock solution satisfying \eqref{2.66} for $t\in[0, +\infty).$

\item[\textcircled{2}] For $\rho_0>0$, if the initial data satisfy one of the following conditions:
  \begin{itemize}
    \item[{\rm vii)}] $[\rho]>0, ~ u_{2}\leq u_{0}\leq u_{1}: ~0<a<\rho_{1}^{2}[u]^{2},~ b\geq-\sqrt{a}$;
    \item[{\rm viii)}] $[\rho]<0, ~ u_{2}\leq u_{0}\leq u_{1}: ~0<a<\rho_{2}^{2}[u]^{2},~ b\geq-\sqrt{a}$,
  \end{itemize}
        then there exists a unique global delta shock solution satisfying \eqref{2.66} for a finite time.
\end{itemize}
\end{proposition}

\begin{proof}
The conclusion about item i) follows from $x'(t)=\frac{u_1+u_2}{2}$ and $u_2<u_1$.

We go on to prove item ii).  By item 2) of Lemma \ref{lem2.3}, we have
$$x'(t)=\frac{[\rho u]+\sqrt{a}}{[\rho]}.$$
Thus
$$x'(t)-u_1=\frac{\rho_2[u]+\sqrt{a}}{[\rho]}=\frac{\rho^2_2[u]^2-a}{[\rho](\rho_2[u]-\sqrt{a})}\leq 0,$$
since $[u]<0,[\rho]>0$ and then $\rho^2_2[u]^2\geq a$.
From $a\geq \rho_{1}^{2}[u]^{2}$,  we have
\begin{equation}\label{2.24} x'(t)-u_2=\frac{a-\rho_{1}^{2}[u]^{2}}{(\sqrt{a}-\rho_1[u])[\rho]}\geq 0.\end{equation}
This completes the proof of item ii).
The proof of item iii) is similar.

Next, we prove item iv).
By item 3) of Lemma \ref{lem2.3}, we know that if $[\rho]=0$ and $[u]<0$, then problem \eqref{1.1}\eqref{1.2} admits a global delta shock solution  given by \eqref{eq44}, i.e.
%
$$x(t)=\frac{\rho_1 [u^{2}]t^{2}-2\rho_{0}u_{0}t}{2(\rho_1 [u]t-\rho_{0})},$$
hence
$$x'(t)=\frac{\rho_1^{2}[u][u^{2}]t^{2}-2\rho_{0}\rho_1 [u^{2}]t+2\rho_{0}^{2}u_{0}}{2(\rho_1 [u]t-\rho_{0})^{2}}.$$
Since $u_{2}\leq u_{0}\leq u_{1},$ we have
\begin{equation}\label{2.22}x'(t)-u_{1}=\frac{\rho_1^{2} [u]^{3}t^{2}-2\rho_{0}\rho_1 [u]^{2}t+2\rho_{0}^{2}(u_{0}-u_{1})}{2(\rho_1[ u]t-\rho_{0})^{2}}\leq 0\end{equation}
and
\begin{equation}\label{2.23}x'(t)-u_{2}= \frac{-\rho_1^{2} [u]^{3}t^{2}+2\rho_{0}\rho [u]^{2}t+2\rho_{0}^{2}(u_{0}-u_{2})}{2(\rho_1 [ u]t-\rho_{0})^{2}}\geq 0\end{equation}
for all $t\geq 0$.
Thus, the entropy condition \eqref{2.66} is satisfied. The proof of item iv) is completed.

Next we turn to the proof of item v). 
In view of $u_{2}\leq u_{0}\leq u_{1}$, we have $$b=[\rho]u_{0}-[\rho u]=\rho_{2}(u_{0}-u_{2})+\rho_{1}(u_{1}-u_{0})\geq 0.$$
By item 4) of Lemma \ref{lem2.3}, 
the delta shock front is
$$x(t)=\frac{[\rho u]t-\rho_{0}+\sqrt{\Delta}}{[\rho]},\quad  t\in [0,+\infty).$$
So
$$x'(t)=\frac{[\rho u]\sqrt{\Delta}+at+\rho_{0} b}{[\rho]\sqrt{\Delta}},$$
and
 $$x'(t)-u_{1}=\frac{(a-\rho_{2}^{2}[u]^{2})at^{2}+2\rho_{0}b(a-\rho_{2}^{2}[u]^{2})t +\rho_{0}^{2}(b^{2}-\rho_{2}^{2}[u]^{2})}{[\rho]\sqrt{\Delta}(at+\rho_{0} b-\rho_{2}[u]\sqrt{\Delta})}.
 $$
Since $[\rho]>0$ implies $a\leq \rho_{2}^{2}[u]^{2}$, $b\geq0$ and
\begin{equation}\label{eq53new}
  \rho_{2}^{2}[u]^{2}-b^{2}=(\rho_{2}[u]+b)(\rho_{2}[u]-b)=[\rho](u_{0}-u_{1})(\rho_{2}[u]-b)\geq 0\end{equation}
for $u_2\leq u_0\leq u_1$, we have $x'(t)-u_1\leq 0, \forall t\ge0$. 
{
Directive calculations give
\begin{equation}\label{2.230}
x'(t)-u_2=\frac{(a-\rho_{1}^{2}[u]^{2})at^{2}+2\rho_{0}b(a-\rho_{1}^{2}[u]^{2})t +\rho_{0}^{2}(b^{2}-\rho_{1}^{2}[u]^{2})}{[\rho]\sqrt{\Delta}(at+\rho_{0} b-\rho_{1}[u]\sqrt{\Delta})}.
\end{equation}
Since $a\geq \rho_{1}^{2}[u]^{2}$, and
$b^{2}\geq\rho_{1}^{2}[u]^{2},$
we have  $x'(t)-u_{2}\geq 0$ for all $t\geq 0$.  Item v) is proved.
The proof of item vi) is similar.} 

{We go on to prove item vii). Due to  $ 0<a<\rho_{1}^{2}[u]^{2}$ and $b\geq-\sqrt{a}$, it follows from \eqref{2.230} that 
$u_{2}<x'(t)$ only holds for $t\in [0,\ t^{\star}),$ where $t^{\star}=\frac{-\rho_{0}b}{a}-\frac{\rho_{0}\rho_{1}[u]}{a}\sqrt{\frac{b^{2}-a}{\rho_{1}^{2}[u]^{2}-a}}$ satisfying $x'(t^{\star})=u_{2}$ and $w_{\rho}(t^{\star})={\rho_{0}}\sqrt{\frac{b^{2}-a}{\rho_{1}^{2}[u]^{2}-a}}\geq 0$. That is to say, the entropy condition \eqref{2.66} is only valid for $t\in [0,\ t^{\star}]$.  Item vii) is proved.
The proof of item viii) is similar.}\qed
\end{proof}

\begin{remark}\label{rm34}
 It follows from the first equality of \eqref{2.24} that the global delta shock solution in item 2) of Lemma \ref{lem2.3} does not satisfy the entropy condition \eqref{2.66} if the initial data satisfy  $[\rho]>0, \ [u]\leq 0, ~0<a<\rho_{1}^{2}[u]^{2}$. So do those  initial data such that  $[\rho]<0, \ [u]\leq 0,~ 0<a<\rho_{2}^{2}[u]^{2}$.\qed
\end{remark}
\begin{remark}\label{rm33}
To illustrate the initial data for the case $\rho_{0}=0$ that the Riemann problem admits a global delta shock satisfying \eqref{2.66}, we introduce two curves by the equation $a=\min\{\rho^{2},\ \rho_{1}^{2}\}[u]^{2}$:
 \begin{equation}\label{2.231}
D_{1}\doteq\left\{(u,\rho)~|~
u=u_{1}-\sqrt{\frac{p-p_{1}}{\rho_1}}=u_1-\sqrt{\frac{\rho^{\gamma} -\rho_{1}^{\gamma}}{\rho_1}}\doteq d_{1}(\rho; U_{1}), \ \  \rho>\rho_{1}\right\},
\end{equation}
\begin{equation}\label{2.232}
D_{2}\doteq\left\{(u,\rho)~|~ u=u_{1}-\sqrt{\frac{p_{1}-p}{\rho}}=u_{1}-\sqrt{\frac{\rho_{1}^{\gamma} -\rho^{\gamma}}{\rho}}\doteq d_{2}(\rho; U_{1}), \ \ \rho<\rho_{1}\right\}.
\end{equation}
The curve  $D_1$ lies below $S_1$ as shown in Figure \ref{fig2}. This follows from \eqref{2.231} and \eqref{3.11}, since
\begin{equation}\begin{split}
\sqrt{\frac{(\rho-\rho_{1})(p-p_{1})}{\rho\rho_1}}-\sqrt{\frac{p-p_{1}}{\rho_1}}&
=\frac{\frac{(\rho-\rho_{1})(p-p_{1})}{\rho\rho_1}-\frac{p-p_{1}}{\rho_1}}{\sqrt{\frac{(\rho-\rho_{1})(p-p_{1})}{\rho\rho_1}}+\sqrt{\frac{p-p_{1}}{\rho_1}}}
 \\&=\frac{-(p-p_{1})}{\rho(\sqrt{\frac{(\rho-\rho_{1})(p-p_{1})}{\rho\rho_1}}+\sqrt{\frac{p-p_{1}}{\rho_1}})}<0,  \ \ \ \  \forall \rho>\rho_{1}.
\end{split}
\end{equation}
Similarly, $D_2$ lies above $S_2$. Then $D_1$ and $D_2$ divide the region $\mathrm{IV}$ into three parts:
\begin{center}
$\mathrm{IV}_{0}\doteq$ the region bounded by curves $D_{1}$ and $D_{2}$,
\end{center}
\begin{center}
$\mathrm{IV}_{1}\doteq$ the region bounded by curves $D_{1}$ and $S_{1}$,
\end{center}
\begin{center}
$\mathrm{IV}_{2}\doteq$ the region bounded by curves $D_{2}$ and $S_{2}$.
\end{center}
{Both $D_1$ and $D_2$ are tangential to the line $\{\rho=\rho_1\}$ at the point $U_1$.}

By Proposition \ref{prp1}, there is a unique global delta shock solution satisfying the over-compressing entropy condition \eqref{2.66} with initial weight of mass $\rho_0=0$, for right state $U_2$ lies in the subregion $\mathrm{IV_0}$. The corresponding classical Riemann solution consists of two shocks, each of which is rather strong, in the sense that $U_2$ is separated from the shock curves $S_1\cup S_2$ by the curves $D_1$ and $D_2$. \qed
\end{remark}

\begin{figure}[t]
\centering
  \includegraphics[width=0.6\textwidth]{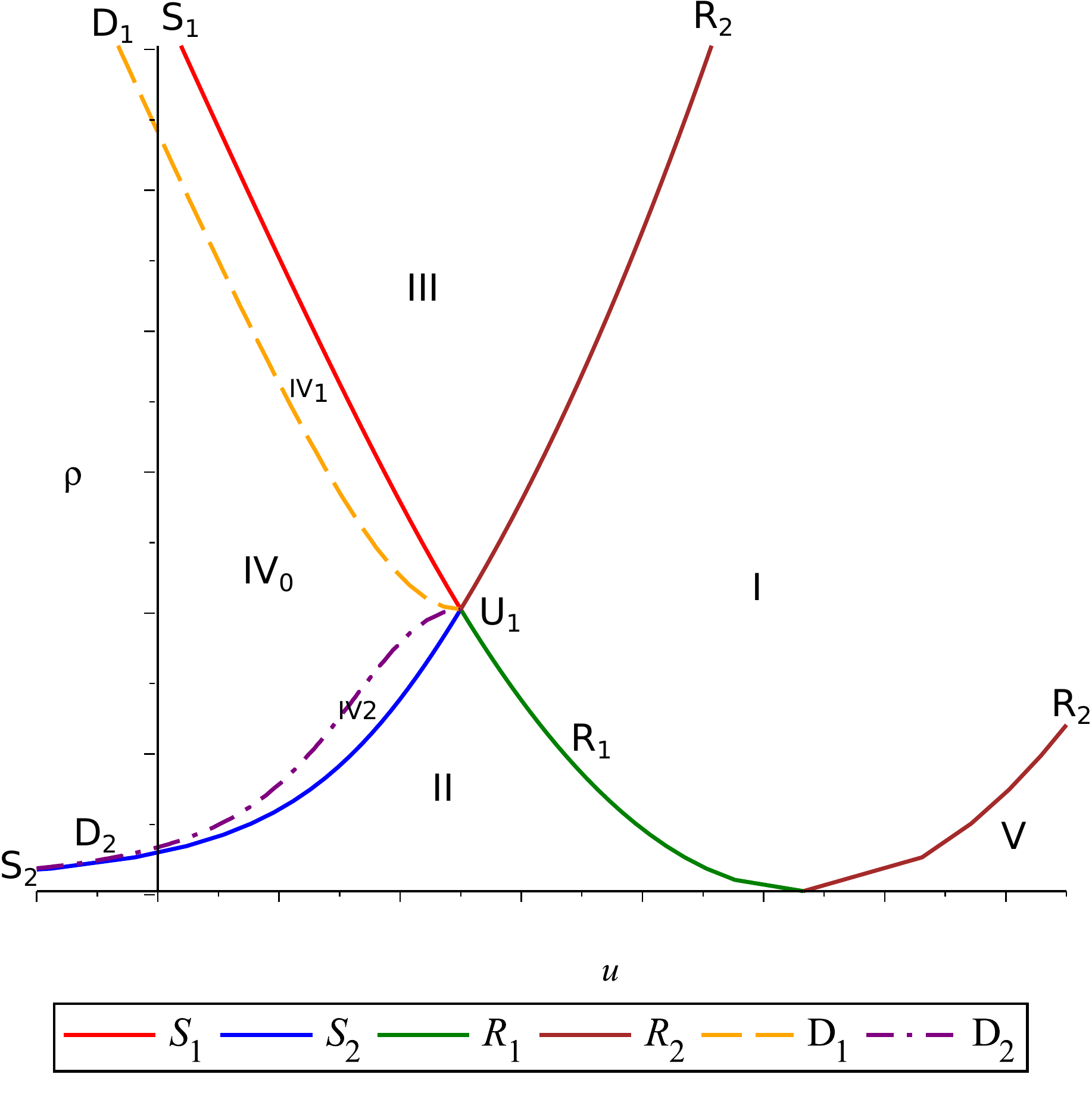}
\caption{
For ${\rho_0=0,}~\mathrm{IV}=\mathrm{IV}_1\cup\mathrm{IV}_0\cup\mathrm{IV}_2$, and for $U_2=(u_2,\rho_2)\in\mathrm{IV}_0$, the global delta shock connecting it with $U_1=(u_1,\rho_1)$ satisfies the over-compressing entropy condition for all the time.}\label{fig2}
\end{figure}

For the local delta shock solution obtained in Lemma \ref{lem2.3}, we have the following conclusions.

\begin{proposition}\label{prp2}
  Consider the Riemann problem \eqref{1.1}\eqref{1.2} for polytropic gases. If the initial data satisfy  $\rho_0>0$ and one of the following conditions:
  \begin{itemize}
    \item[1)]$[\rho]>0, \ u_{2}\leq u_{0}\leq u_{1}, \ 0>a\geq -\rho_{1}^{2}[u]^{2}$;
    \item[2)]$[\rho]<0, \ u_{2}\leq u_{0}\leq u_{1}, \ a\leq-\rho_{1}^{2}[u]^{2}$;
    \item[3)]$[\rho]<0, \ u_{2}\leq u_{0}\leq u_{1}, \ 0>a\geq -\rho_{2}^{2}[u]^{2}$;
    \item[4)]$[\rho]>0, \ u_{2}\leq u_{0}\leq u_{1}, a\leq-\rho_{2}^{2}[u]^{2}$,
  \end{itemize}
then there exists a local delta shock solution and it satisfies \eqref{2.66} for a shorter time.
\end{proposition}
\begin{proof}
In view of $u_2\leq u_0\leq u_1,$ we have $b\geq 0.$
By item 4) of Lemma \ref{lem2.3}, for
all these items above, there  exist local delta shock solutions given by
$$x(t)=\frac{[\rho u]t-\rho_{0}+\sqrt{\Delta}}{[\rho]},\quad  t\in [0,t^{*}).$$
Then
 $$x'(t)-u_{1}=\frac{(a+\rho_{2}^{2}[u]^{2})at^{2}+2\rho_{2}[u](\rho_{0}\rho_{2}b[u]+a\sqrt{\Delta})t +(\rho_{2}^{2}[u]^{2}-b^{2})\rho_{0}^{2}}{[\rho]\sqrt{\Delta}(\rho_{2}[u]\sqrt{\Delta}+at-\rho_{0}b)}.$$
We now prove item 1). Recall that we have
$\rho_{2}^{2}[u]^{2}-b^{2}\geq 0$ by \eqref{eq53new}.
In view of  $a<0, \ a+\rho_{1}^{2}[u]^{2}\geq 0$ and $[u]<0$, it holds 
\begin{equation}\label{3.55}(a+\rho_{2}^{2}[u]^{2})at^{2}+2\rho_{2}[u](\rho_{0}\rho_{2}b[u]+a\sqrt{\Delta})t +(\rho_{2}^{2}[u]^{2}-b^{2})\rho_{0}^{2}\geq 0 \end{equation}
only for $t\in [0,\ t^{**}),$ for some  $t^{**}<t^{*}.$
Therefore, $x'(t)-u_{1}\leq 0$ for $t\in [0,\ t^{**}).$  Similarly, $$x'(t)-u_{2}=\frac{(a+\rho_{1}^{2}[u]^{2})at^{2}+2\rho_{1}[u](\rho_{0}\rho_{1}b[u]+a\sqrt{\Delta})t +(\rho_{1}^{2}[u]^{2}-b^{2})\rho_{0}^{2}}{[\rho]\sqrt{\Delta}(\rho_{1}[u]\sqrt{\Delta}+at-\rho_{0}b)}.$$
One checks that $\rho_{1}^{2}[u]^{2}\leq b^{2}$ under the assumptions of item 1). Thus
we also have $x'(t)-u_2\geq 0$ for $t\in [0,t_{**})$, where $t_{**}<t^{*}.$ So the local delta shock solutions obtained by Lemma \ref{lem2.3} satisfy entropy condition \eqref{2.66} for $t\in [0,\ \min\{t^{**}, \ t_{**}\}).$ Item 1) is proved.

The proofs of items 2)-4) are similar, and thus omitted.\qed 
\end{proof}

\begin{proposition}\label{prp3}  For the delta shock solution obtained in Lemma \ref{lem2.3} and  satisfying \eqref{2.66}, 
the initial data affect the profiles of the solutions in the following way:
\begin{itemize}
  \item[$\alpha)$] for $\rho_{0}=0$, each solution obtained in Lemma \ref{lem2.3} is self-similar; 
  \item[$\beta)$] for $\rho_{0}>0$,
\begin{itemize}
  \item[a)] the delta shock front is convex if the initial data satisfy in addition one of the following three conditions: ~i)~ 
      $[\rho]=0$ and $u_2\leq u_0<\frac{u_{1}+u_{2}}{2}$; ~ii)~ $[\rho]\ne0$ and 
      $u_2\leq u_{0}<\min\{u_1, \frac{[\rho u]+\sqrt{a}}{[\rho]}\}$;
  \item[b)] the delta shock front is concave if the initial data satisfy  in addition one of the following three conditions: ~i)~ 
      $[\rho]=0$ and $u_1\geq u_0>\frac{u_{1}+u_{2}}{2}$; ~ii)~ $[\rho]\ne0$  and 
      $u_1\geq u_{0}>\max\{u_2, \frac{[\rho u]+\sqrt{a}}{[\rho]}\}.$
\end{itemize}
\end{itemize}
\end{proposition}
\begin{proof}
Item $\alpha$) follows directly  from the expression of $x(t)$ in items 1) and 2) of Lemma \ref{lem2.3}.

We now prove item $\beta$). For $\rho_{0}>0$, if $[\rho]=0$, we have $x''(t)=\frac{\rho_1\rho_{0}^{2}[u](u_{1}+u_{2}-2u_{0})}{(\rho_1 [u]t-\rho_{0})^{3}}$. Since $[u]<0$, it is obvious that
\begin{equation*}
x''(t)\left \{
\begin{split}
&>0, \ \ \ &{\rm if~}u_{0}<\frac{u_{1}+u_{2}}{2}, \\
&=0, \ \ \ &{\rm if~}u_{0}=\frac{u_{1}+u_{2}}{2}, \\
&<0, \ \ \ &{\rm if~}u_{0}>\frac{u_{1}+u_{2}}{2}.
\end{split} \right.
\end{equation*}
This completes the proof of the cases i) of a) and i) of b) in item $\beta$).

If $[\rho]\neq0$, we have { $x''(t)
=\frac{\rho_{0}^{2}(\sqrt{a}+b)(\sqrt{a}-b)}{[\rho]\Delta^{3/2}}=\frac{\rho_{0}^{2}(\sqrt{a}+b)}{[\rho]\Delta^{3/2}}(\sqrt{a}+[\rho u]-[\rho]u_0)$. Due to $u_2<u_0<u_1$, and then $b\geq 0$, we have 
\begin{equation*}
x''(t)\left \{
\begin{split}
&>0, \quad{\rm if~}
u_{0}<\frac{[\rho u]+\sqrt{a}}{[\rho]}, \\
&=0, \quad{\rm if~}u_{0}=\frac{[\rho u]+\sqrt{a}}{[\rho]}, \\
&<0, \quad{\rm if~}
u_{0}>\frac{[\rho u]+\sqrt{a}}{[\rho]}.
\end{split} \right.
\end{equation*}
Combining it with the assumption $u_2\leq u_0\leq u_1$, we complete the proofs of the other items.}\qed 
\end{proof}

\section{Radon measure solutions to Riemann problems}\label{s3}
Based upon analysis in the previous section on a single delta shock solution, we now study the existence of (Radon measure) solutions to the Riemann problems  \eqref{1.1}\eqref{1.2}. There are two cases: for $\rho_0=0$, we call it a classical Riemann problem; for $\rho_0>0$, we call it a generalized or singular Riemann problem.

\subsection{The classical Riemann problems}\label{s3.1}

The classical Riemann solutions consist of only two kinds of elementary waves: rarefaction waves and shock waves, and the latter correspond to delta shocks with trivial weight of mass. We now focus on the existence of Radon measure solution consisting of delta shocks with positive mass weights. As before, we denote $(u_{i},\rho_{i}) $ by $U_{i}\ (i=1,2).$


\begin{theorem}\label{thm3.1}
Given $U_{1}$, 
for the Riemann problem \eqref{1.1}\eqref{1.01}, we have the following conclusions:
  \begin{itemize}
    \item[1)] 
    there is a unique global solution consisting of a single delta shock if and only if $U_2\in(\mathrm{I_{0}}\setminus\{(u,\rho)|\rho=\rho_1, u>u_1\})\cup \mathrm{IV}$. However, only those with $U_2\in\mathrm{IV_{0}}$ 
     satisfy the over-compressing condition \eqref{2.66};
    \item[2)] if $U_{2}\in \mathrm{IV}_1\cup \mathrm{III}$, 
        then there are {infinitely global solutions satisfying \eqref{2.66}. Each of these solutions consists} of a delta shock followed by a rarefaction wave $R_2$;
    \item[3)] if $U_{2}\in \mathrm{IV}_2\cup \mathrm{II}$, 
        then there are  infinitely global solutions satisfying \eqref{2.66}. Each consists of a rarefaction wave $R_1$ followed by a delta shock;
    \item[4)] if $U_{2}$ belongs to the other part of the upper-half $(u,\rho)$-plane, 
    then there is no solution consisting of one or two waves (one of which is a delta shock), even for a short time and without the restriction of \eqref{2.66}.
  \end{itemize}
\end{theorem}

\begin{proof}
Item 1) follows directly from Lemma \ref{lem2.3} and Proposition \ref{prp1}.

We now define, for $\overline{U}=(\overline{u}, \ \overline{\rho})\in \Bbb R\times\mathbb{R}^+$, and $i=1,2$,  the curves
\begin{align*}
&D_{i}(\overline{U})\doteq\{(u,\ \rho)~|~ u=d_{i}(\rho; \overline{U}),\ \ u<\overline{u}\},\\
&R_{i}(\overline{U})\doteq\{(u,\ \rho)~|~ u=r_{i}(\rho; \overline{U}),\ \ u>\overline{u}\},\\
&R_{i}^*(\overline{U})\doteq\{(u,\ \rho)~|~ u=r_{i}(\rho; \overline{U}),\ \ u<\overline{u}\},
\end{align*}
where $d_i(\rho;\bar{U})$ are given by \eqref{2.231}\eqref{2.232}, $r_i(\rho;\bar{U})$ by \eqref{3.17}\eqref{3.18}, and for $R_i^{*}(\bar{U})$, see \eqref{eq69}\eqref{eq70}.

For given $U_{1}\in \Bbb R\times \mathbb{R}^{+},$ and any $U_{2}\in \mathrm{IV}_1\cup \mathrm{III}$ (these regions depend on $U_1$), along the curve $R_2^{*}(U_2)$, $u$ is increasing with respect to increasing $\rho$,  and  $\rho$ monotonically decreases to $0$ as $u$ decreases to $-\infty$. Also notice that $u=d_1(\rho,U_1)$ is decreasing with respect to $\rho$, and it monotonically  goes to $-\infty$ as $\rho$ increases to $+\infty$, while $r_2(\rho, U_1)$ is increasing with respect to $\rho$ and goes to $+\infty$ as  $\rho$ goes to $+\infty$. Then
the horizontal line $\{\rho=\rho_{2}\}$ meets $D_{1}(U_{1})$ at a point $A$, and $R_{2}(U_{1})$ at a point $B$ (see Figure \ref{fig4}). By continuity, we know that the curve $ R_{2}^{*}(U_{2})$ with $u<u_2$ will penetrate into the domain $\mathrm{IV}_0$ and thus $R^{*}_2(U_2)\cap \mathrm{IV}_0\neq \emptyset$.
We claim that:
\begin{quote}
($\maltese$):\quad There are infinite $\widetilde{U}\in R_2^{*}(U_2)\cap \mathrm{IV}_0$ such that the delta shock $x=x(t)$ connecting $U_1$ and $\widetilde{U}$ satisfies $\boxed{x'(t)\leq\lambda_2(\tilde{U})}$ and the entropy condition \eqref{2.66}.
\end{quote}
If it is true, then as $\widetilde{U}$ can be connected to $U_2$ by a $2$-rarefaction wave,  the solution to Riemann problem \eqref{1.1}\eqref{1.01} consists of a delta shock followed by a 2-rarefaction wave $R_2$,  written shortly as $\delta+R_2$.  It follows from this claim that there are infinitely global solutions satisfying \eqref{2.66}.

%
\begin{figure}[t]
\centering
  \includegraphics[width=0.6\textwidth]{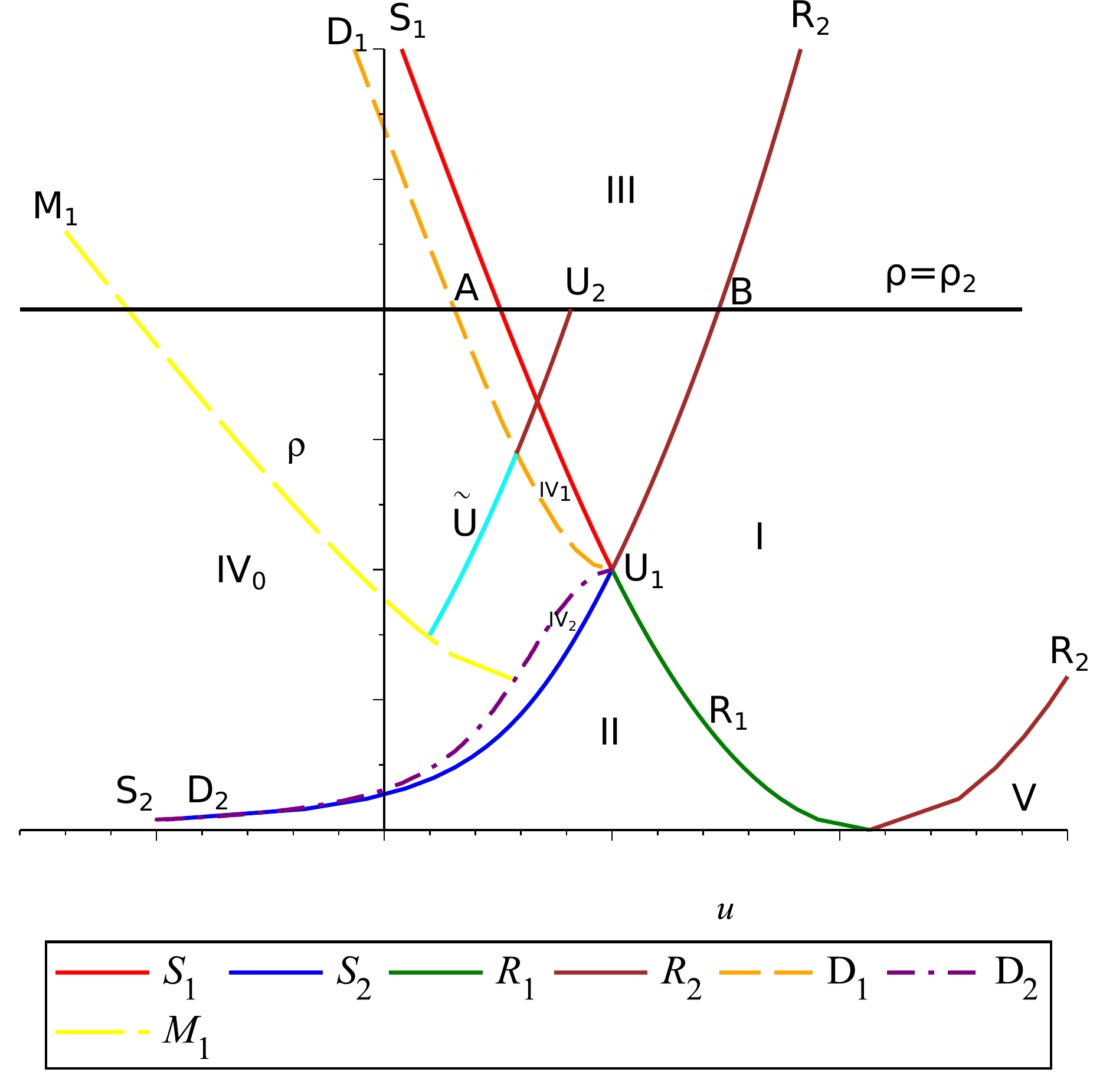}
\caption{The case of $U_{2}$ in region $\mathrm{IV}_1\cup \mathrm{III}$: delta shock followed by a rarefaction wave.  }\label{fig4}
\end{figure}

We are left to prove the claim ($\maltese$). By item 1), for all $\widetilde{U}\doteq(u_m,\rho_m)\in R_2^{*}(U_2)\cap \mathrm{IV}_0$, the delta shock satisfies the entropy condition \eqref{2.66}. So we only need to show that there are infinite $\widetilde{U}$ satisfy $x'(t)\leq\lambda_2(\tilde{U})$. Recalling from Lemma \ref{lem2.3} that  $x'(t)=\frac{[\rho u]+\sqrt{a}}{[\rho]}$, it is reduced to  $\frac{[\rho u]+\sqrt{a}}{[\rho]}\leq u_{m}+c_{m}$ for $u_m\le u_1$. Notice here that for
$[\rho]=\rho_{m}-\rho_{1}>0,$ and $[\rho u]=\rho_mu_m-\rho_1u_1$, this is equivalent to $$\sqrt{a}\leq[\rho](u_{m}+c_{m})-[\rho u]=[\rho]c_{m}-\rho_{1}[u],$$
or, remembering that $a>0$ and $[u]=u_m-u_1<0$ in $\mathrm{IV}_0$, $$a=\rho_{1}\rho_{m}[u]^2-[\rho][p]\leq([\rho]c_{m}-\rho_{1}[u])^2=[\rho]^{2}c_{m}^{2}-2\rho_{1}c_{m}[\rho][u]+\rho_{1}^{2}[u]^2,$$
which is simplified to  $$\rho_{1}[\rho][u]^2-[\rho][p]\leq[\rho]^{2}c_{m}^{2}-2\rho_{1}c_{m}[\rho][u].$$
For $[\rho]>0,$ it reads  \begin{equation}\label{2.992}\rho_{1}[u]^2-[p]\leq[\rho]c_{m}^{2}-2\rho_{1}c_{m}[u].\end{equation}
The solution to \eqref{2.992}, together with $u_m<u_1$, gives
\begin{equation}\label{2.993}u_1-\sqrt{\gamma\rho_{m}^{\gamma-1}}-\sqrt{\frac{(\gamma+1)\rho_{m}^{\gamma}-\rho_{1}^{\gamma}}{\rho_{1}}}\leq u_m\leq  u_1 , \ \ \  \rho_{m}>\rho_{1}.\end{equation}
For $[\rho]=\rho_{m}-\rho_{1}<0$, similar to the case of $[\rho]>0$, we obtain that $\widetilde{U}\in R_2^{*}(U_2)\cap \mathrm{IV}_0$ satisfies $x'(t)\leq\lambda_2(\tilde{U})$ if and only if
\begin{align}\label{2.995}  u_1-\sqrt{\gamma\rho_{m}^{\gamma-1}}-\sqrt{\frac{(\gamma+1)\rho_{m}^{\gamma}-\rho_{1}^{\gamma}}{\rho_{1}}}\leq &u_m\leq u_1-\sqrt{\gamma\rho_{m}^{\gamma-1}}+\sqrt{\frac{(\gamma+1)\rho_{m}^{\gamma}-\rho_{1}^{\gamma}}{\rho_{1}}},\nonumber \\ \frac{\rho_1}{(\gamma+1)^{1/\gamma}}\leq&\rho_m<\rho_1.\end{align}
For $[\rho]=0$, we require that $\frac{u_{1}+u_{m}}{2}\leq u_{m}+c_{m}$, i.e.,  \begin{equation}\label{2.994} u_1\geq u_m\geq u_1-2c_m=u_1-2c_1.\end{equation}
It is observed that $u_1-\sqrt{\gamma\rho_{m}^{\gamma-1}}-\sqrt{\frac{(\gamma+1)\rho_{m}^{\gamma}-\rho_{1}^{\gamma}}{\rho_{1}}}\rightarrow u_1-2c_1$ as $\rho_{m}\rightarrow\rho_{1},$  and $u_1-\sqrt{\gamma\rho_{m}^{\gamma-1}}+\sqrt{\frac{(\gamma+1)\rho_{m}^{\gamma}-\rho_{1}^{\gamma}}{\rho_{1}}}\rightarrow u_1$ as $\rho_{m}\rightarrow\rho_{1}.$ Therefore, \eqref{2.993}, \eqref{2.995} and \eqref{2.994} are consistent.

To clarify the region defined by \eqref{2.993}, \eqref{2.995} and \eqref{2.994}, it is natural to define the curves
\begin{align}
&M_1\doteq\left\{(u,\rho)~|~ u=u_1-\sqrt{\gamma\rho^{\gamma-1}}-\sqrt{\frac{(\gamma+1)\rho^{\gamma} -\rho_{1}^{\gamma}}{\rho_{1}}}, \ \ \rho\ge \frac{\rho_1}{(\gamma+1)^{1/\gamma}}\right\},\\
&M_2\doteq\left\{(u,\rho)~|~ u=u_1-\sqrt{\gamma\rho^{\gamma-1}}+\sqrt{\frac{(\gamma+1)\rho^{\gamma} -\rho_{1}^{\gamma}}{\rho_{1}}}, \ \ \frac{\rho_1}{(\gamma+1)^{1/\gamma}}\leq\rho\leq\rho_1\right\}.\label{2.996}\end{align}

We now prove that $M_2$ lies below $D_2$, hence it is not drawn in Figure \ref{fig4}. To this end, by \eqref{2.996} and \eqref{2.232}, we need to show that
\begin{equation}\label{2.998} \begin{split}
&\sqrt{\frac{(\gamma+1)\rho^{\gamma} -\rho_{1}^{\gamma}}{\rho_{1}}}-\sqrt{\gamma\rho^{\gamma-1}}+\sqrt{\frac{\rho_{1}^{\gamma}-\rho^{\gamma}}{\rho}}
\\=&\frac{\{(\gamma+1)\rho^{\gamma}-\rho_{1}^{\gamma}\}(\rho-\rho_1)+2\rho\rho_1\sqrt{\frac{\{(\gamma+1)\rho^{\gamma} -\rho_{1}^{\gamma}\}(\rho_{1}^{\gamma}-\rho^{\gamma})}{\rho\rho_{1}}}}{\rho\rho_{1}\left(\sqrt{\frac{(\gamma+1)\rho^{\gamma} -\rho_{1}^{\gamma}}{\rho_{1}}}+\sqrt{\frac{\rho^{\gamma}-\rho_{1}^{\gamma}}{\rho}}+\sqrt{\gamma\rho^{\gamma-1}}\right)}>0.
\end{split}\end{equation}
Since the denominator is positive, it is equivalent to prove the numerator is positive.
That is,
\begin{equation}\label{2.997} \begin{split}
K(\rho)&\doteq\{(\gamma+1)\rho^{\gamma}-\rho_{1}^{\gamma}\}(\rho-\rho_1)^{2}-4\rho\rho_1(\rho_{1}^{\gamma}-\rho^{\gamma})
\\&=\gamma\rho^{\gamma}(\rho-\rho_1)^{2}+(\rho+\rho_1)^{2}(\rho^{\gamma}-\rho_{1}^{\gamma})<0.
\end{split}\end{equation}
In fact, it follows from the fact that, for $\frac{\rho_1}{(\gamma+1)^{1/\gamma}}\leq\rho\leq\rho_1,$ there holds
\begin{align}\label{2.9997}
K'(\rho)&=\gamma^{2}\rho^{\gamma-1}(\rho-\rho_1)^{2}+2\gamma\rho^{\gamma}(\rho-\rho_1)\nonumber\\&\qquad+\gamma\rho^{\gamma-1}(\rho+\rho_1)^{2}+2(\rho^{\gamma}-\rho_{1}^{\gamma})(\rho+\rho_1)
\nonumber\\&=\gamma(\gamma+1)\rho^{\gamma-1}(\rho-\rho_1)^{2}+2\{(\gamma+1)\rho^{\gamma}-\rho_{1}^{\gamma}\}(\rho+\rho_1)>0,
\end{align}
and $K(\rho_1)=0.$

In summary, for given $U_1$, and any $U_{2}\in\mathrm{IV}_1\cup \mathrm{III}$, $U_1$ can be connected to any $\widetilde{U}=(u_{m},\rho_{m})\in (R_2^{*}(U_2)\cap \mathrm{IV}_0\cap Q_1)\subset \mathrm{IV}_0$ by a delta shock, while $\widetilde{U}$ is connected to $U_2$ by a 2-rarefaction wave, where \begin{align}\label{2.988} Q_1\doteq&\left\{(u,\rho)~|~u_1-\sqrt{\gamma\rho^{\gamma-1}} -\sqrt{\frac{(\gamma+1)\rho^{\gamma} -\rho_{1}^{\gamma}}{\rho_{1}}}\leq u\leq u_1,\ \rho\geq\rho_1\right\}\nonumber\\
&\bigcup \left\{(u,\rho)~|~u_1-\sqrt{\gamma\rho^{\gamma-1}} -\sqrt{\frac{(\gamma+1)\rho^{\gamma} -\rho_{1}^{\gamma}}{\rho_{1}}}\leq u\leq u_1-\sqrt{\gamma\rho^{\gamma-1}} \right.\nonumber\\&\qquad\left.+\sqrt{\frac{(\gamma+1)\rho^{\gamma} -\rho_{1}^{\gamma}}{\rho_{1}}} ,\ \frac{\rho_1}{(\gamma+1)^{1/\gamma}}\leq\rho\leq\rho_1\right\}. \end{align} 
In Figure \ref{fig4},  $\widetilde{U}$ is taken on the part of the curve $R_2^{*}(U_2)$ in the region surrounded by $M_1,  \ D_2 , \ D_1$ and the line $\{\rho=\rho_2\}$. Obviously there are infinite such $\widetilde{U}$.  This proves the claim $(\maltese)$ and hence item 2).
\smallskip

The proof of item 3) is similar to that of item 2), see Figure \ref{fig41}. For fixed $U_2\in\mathrm{II}\cup\mathrm{IV}_2$, we show that there are infinite $\widetilde{U}=(u_m,\ \rho_{m})\in R_1(U_1)$, such that the delta shock connecting $\tilde{U}$ and $U_2$ satisfies $x'(t)\geq \lambda_{1}(\widetilde{U})=u_m-c_m$ and the entropy condition \eqref{2.66}.
\begin{figure}[t]
\centering
  \includegraphics[width=0.6\textwidth]{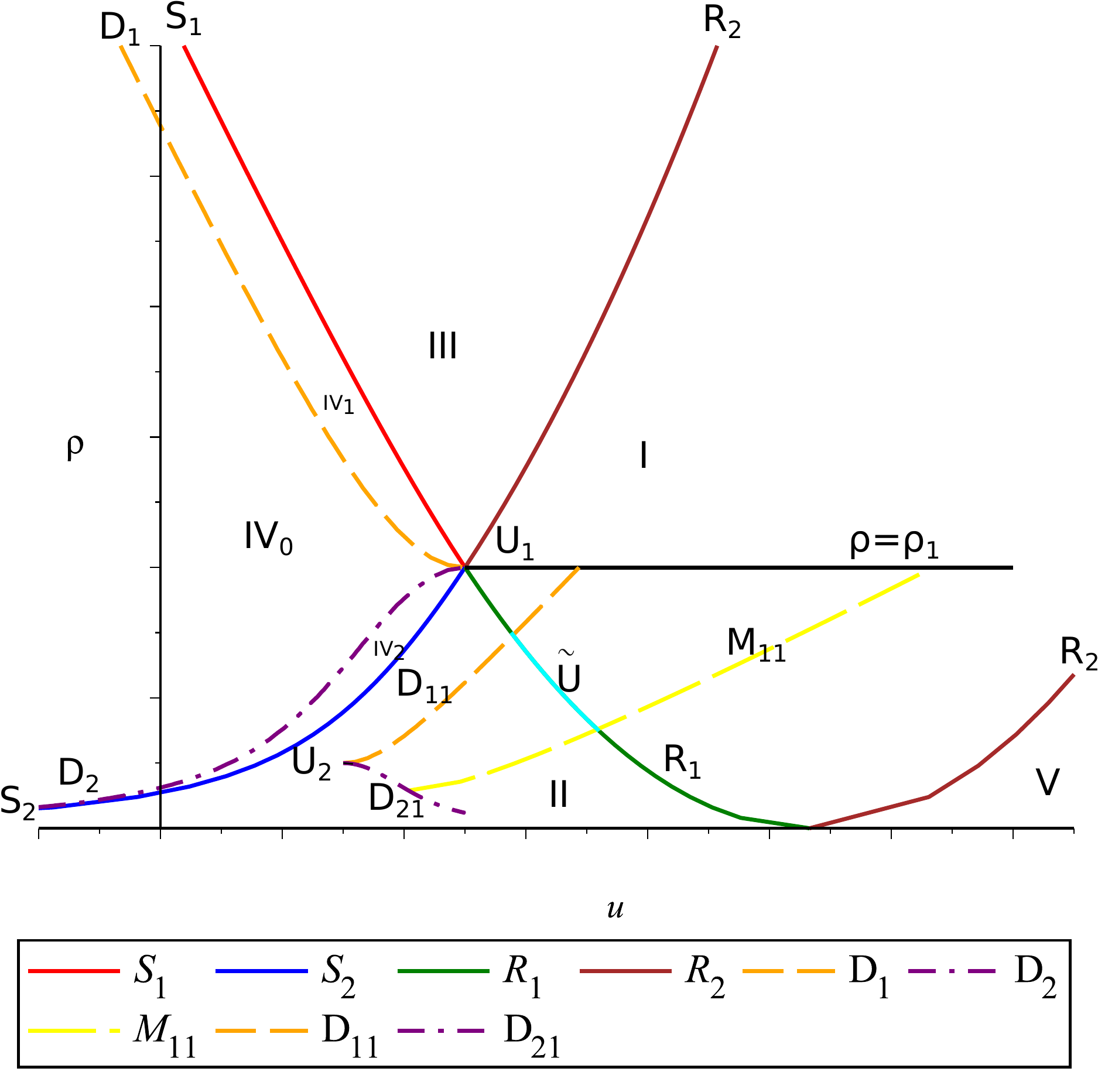}
\caption{$U_{2}$ in region $\mathrm{IV}_2\cup \mathrm{II}$: a 1-rarefaction wave followed by a delta shock. }\label{fig41}
\end{figure}

For $\widetilde{U}=(u_m, \rho_{m})$ satisfying $x'(t)=\frac{[\rho u]+\sqrt{a}}{[\rho]}\geq u_m-c_m$, we find that
 \begin{equation}\label{2.98}\begin{split}
u_2+\sqrt{\gamma\rho_m^{\gamma-1}}-\sqrt{\frac{(\gamma+1)\rho_m^{\gamma} -\rho_{2}^{\gamma}}{\rho_{2}}}\leq u_m\leq u_2+\sqrt{\gamma\rho_m^{\gamma-1}} +\sqrt{\frac{(\gamma+1)\rho_m^{\gamma}-\rho_{2}^{\gamma}}{\rho_{2}}}
\end{split} \end{equation}
if $\frac{\rho_2}{(\gamma+1)^{1/\gamma}}\leq\rho_m\leq\rho_2;$ and
 \begin{equation}\label{2.9871}\begin{split}
u_2 \leq u_m \leq u_2+\sqrt{\gamma\rho_m^{\gamma-1}} +\sqrt{\frac{(\gamma+1)\rho_m^{\gamma} -\rho_{2}^{\gamma}}{\rho_{2}}} \end{split} \end{equation}
if  $\rho_m\geq\rho_2$.
The entropy condition \eqref{2.66} requires that
  \begin{equation}\label{2.981}\begin{split}
u_m \geq u_2+\sqrt{\frac{\rho_m^{\gamma} -\rho_{2}^{\gamma}}{\rho_{2}}}, \qquad\text{for}~~ \rho_m\geq\rho_2,
\end{split} \end{equation}
and
 \begin{equation}\label{2.987}\begin{split}
u_m \geq u_2+\sqrt{\frac{\rho_2^{\gamma} -\rho_{m}^{\gamma}}{\rho_{m}}}, \qquad\text{for}~~ \rho_m\leq\rho_2. \end{split} \end{equation}
Hence we define the curves
\begin{align*}
  M_{11}&\doteq\left\{(u,\rho)~|~u=u_2+\sqrt{\gamma\rho^{\gamma-1}}+\sqrt{\frac{(\gamma+1)\rho^{\gamma}-\rho_{2}^{\gamma}}{\rho_{2}}}, \ \rho \geq \frac{\rho_2}{(\gamma+1)^{1/\gamma}}\right\},\\
  M_{21}&\doteq\left\{(u,\rho)~|~ u=u_2+\sqrt{\gamma\rho^{\gamma-1}}-\sqrt{\frac{(\gamma+1)\rho^{\gamma} -\rho_{2}^{\gamma}}{\rho_{2}}}, \  \frac{\rho_2}{(\gamma+1)^{1/\gamma}}\le\rho\leq\rho_2\right\},\\
  D_{21}&\doteq\left\{(u,\rho)~|~ u= u_2+\sqrt{\frac{\rho_2^{\gamma} -\rho^{\gamma}}{\rho}}, \ \rho\leq\rho_2\right\}, \\
  D_{11}&\doteq\left\{(u,\rho)~|~u= u_2+\sqrt{\frac{\rho^{\gamma} -\rho_{2}^{\gamma}}{\rho_2}}, \ \rho_1>\rho\geq\rho_2\right\}.
\end{align*}
Similar to \eqref{2.998}, one can show that $M_{21}$ is on the left-hand-side of $D_{21}$, hence not drawn in Figure \ref{fig41}.

Therefore,  for given $U_1$ and any $U_{2}\in\mathrm{IV}_2\cup \mathrm{II}$, $U_1$ can be connected to $\widetilde{U}=(u_m,\ \rho_{m})\in R_1(U_1)\cap Q_2$ by a $1$-rarefaction wave, while $\widetilde{U}$ is connected to $U_2$ by a delta shock satisfying $x'(t)\geq \lambda_{1}(\widetilde{U})$ and the entropy condition \eqref{2.66}, where
\begin{align}\label{2.9888} Q_2=&\Bigg\{(u,\rho)|u_2+\sqrt{\frac{\rho^{\gamma} -\rho_{2}^{\gamma}}{\rho_{2}}}\leq u\leq u_2+\sqrt{\gamma\rho^{\gamma-1}} +\sqrt{\frac{(\gamma+1)\rho^{\gamma} -\rho_{2}^{\gamma}}{\rho_{2}}},~\rho_2\leq\rho\nonumber\\ &\leq\rho_1\Bigg\}
\bigcup  \Bigg\{(u,\rho)|\ u_2+\sqrt{\frac{\rho_{2}^{\gamma}-\rho^{\gamma}}{\rho}}\leq u\leq u_2+\sqrt{\gamma\rho^{\gamma-1}} +\sqrt{\frac{(\gamma+1)\rho^{\gamma} -\rho_{2}^{\gamma}}{\rho_{2}}},\nonumber\\
&\qquad\frac{\rho_2}{(\gamma+1)^{1/\gamma}}\leq\rho\leq\rho_2\Bigg\}. \end{align}
Thus, as shown in Figure \ref{fig41}, there are infinite $\widetilde{U}$ on $R_1(U_1)$, lying in the region bounded by $M_{11}$, $D_{11}$, $D_{21}$ and $\{\rho=\rho_1\}$, and fulfilling the requirements.  This proves item 3).

We turn to the proof of item 4). Let 
$U_2$ lie in the region $\mathrm{I}$.
Proposition \ref{prp1} indicates that
$U_1$ and $U_2$ cannot be connected by a single delta shock satisfying \eqref{2.66}. We claim that there is neither a solution of type $\delta+R_2$, nor  of type $R_1+\delta$. 
We prove below that $\delta+R_2$ is impossible. Then by symmetry, the case $R_1+\delta$ is also impossible.

Since the 2-rarefaction wave curves cover the region $\mathrm{I}$ and are not intersecting,  by the fact that if $U'\in R_2(U'')$, then $U''\in R^{*}_2(U')$, it is enough to show that
the curve $R_2^{*}(U_1)$ 
locates below $D_2$. 
To demonstrate this, noticing that similar to the verification of \eqref{eq49new}, where we showed that $S_{22}$ is on the right-hand-side of $R_2$, we know that
$R_2^{*}(U_1)$ lies below $S_2$, as shown in Figure \ref{fig7}. Then recalling that we have shown that $D_2$ locates above $S_2$, we have the desired conclusion. 
\begin{figure}[t]
\centering
  \includegraphics[width=0.6\textwidth]{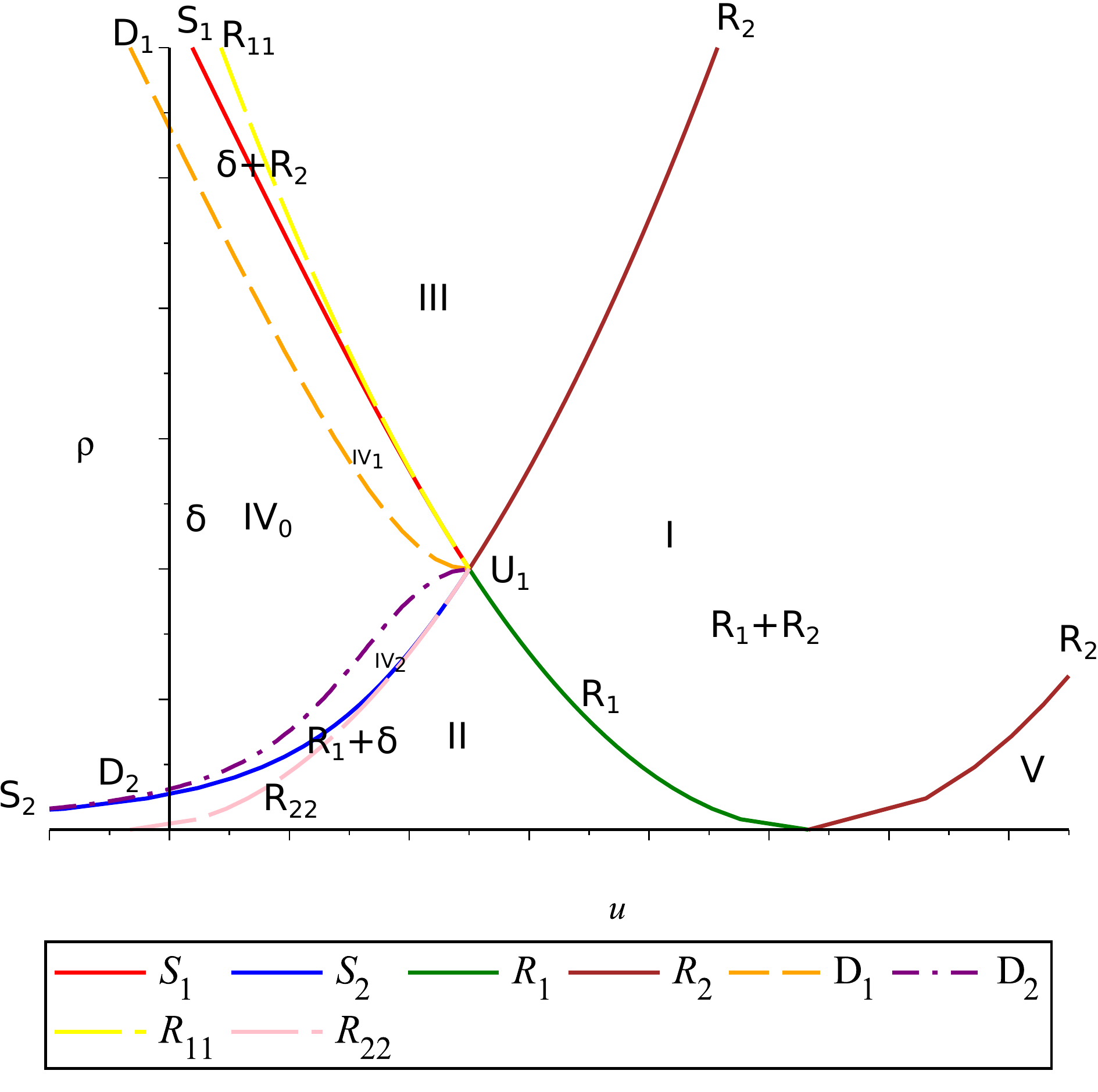}
\caption{Solutions to Riemann problem for $\rho_{0}=0$.}\label{fig7}
\end{figure}

Obviously, for $U_2$ in region $\mathrm{V}$, there is also no solution that contains a delta shock satisfying the over-compressing entropy condition (cf. the proof of item 3) above). The proof of item 4) is completed.\qed
\end{proof}
\vspace{-0.7cm}
\begin{remark}\label{rem4}
We summarize the new solutions to \eqref{1.1}\eqref{1.01} and  compare them   with the classical Riemann solutions in Table \ref{tab1}. However, back to item 1) of Theorem \ref{thm3.1}, even for $U_{2}\in \mathrm{IV}_0$, if we discard the principle of finding solutions with minimal number of waves,  it is easy to see that there are infinite solutions consisting of $\delta+R_{2}$, $R_{1}+\delta$, or two delta shocks, which {\it  do  fulfill} the over-compressing entropy condition.\qed

\begin{table}[t]
\centering{\small
\caption{Solutions to \eqref{1.1}\eqref{1.01} for given left-state $U_1$ and variable right-state $U_2$.}\label{tab1}
  \begin{tabular}{|p{2.7em}|p{2.3em}|p{12em}|p{14em}|}
\hline
\multicolumn{2}{|l|}{$U_2$ }  &New solutions & Classical Riemann solvers\\
 \hline
 \multicolumn{2}{|l|}{{in region $\mathrm{IV}_{0}$}}& a single delta shock ($\delta$) & $S_{1}+S_{2}$ \\
\hline
\multicolumn{2}{|l|}{{in region $\mathrm{IV}_{1}$}}& $\delta+R_{2}$ &  $S_{1}+S_{2}$  \\
\hline
\multicolumn{2}{|l|}{{in region $\mathrm{IV}_{2}$}}& $R_{1}+\delta$ & $S_{1}+S_{2}$\\
\hline
\multicolumn{2}{|l|}{{in region $\mathrm{III}$}}& $\delta+R_{2}$ & $S_{1}+R_{2}$\\
\hline
\multicolumn{2}{|l|}{{in region $\mathrm{II}$}}& $R_{1}+\delta$ & $R_{1}+S_{2}$\\
\hline
\multicolumn{2}{|l|}{{in region $\mathrm{I}$}}& none & $R_{1}+R_{2}$\\
\hline
 \multicolumn{2}{|l|}{{in region $\mathrm{V}$}}& none & $R_{1}+\mathrm{Vacuum}+R_{2}$\\
\hline
\end{tabular}}
\end{table}
\end{remark}
\vspace{-0.5cm}
\subsection{The singular Riemann problems}\label{s3.2}
Recall that for given initial data \eqref{1.2}, we have set $a\doteq [\rho u]^{2}-[\rho][\rho u^{2}+p]=\rho_{1}\rho_{2}[u]^{2}-[\rho][p],~b\doteq [\rho]u_{0}-[\rho u]. $
\begin{theorem}\label{thm4.1}
  For the generalized Riemann problem \eqref{1.1}\eqref{1.2} with $\rho_0>0$, we have the following conclusions about the existence of solutions:
   \begin{itemize}
   \item[1)] if the initial data satisfy one of the following: ~i)~ $a>0,~ b>-\sqrt{a}$; ~ii)~ $a=0, ~ b\geq0$; ~iii)~ $[\rho]\neq0,~ a>0,~ b=-\sqrt{a}$, then there is a unique global solution that the left and right states can be connected by a delta shock;
   \item[2)] if the initial data satisfy one of the following: ~i)~ $a>0, ~b<-\sqrt{a}$; ~ii)~ $a\leq 0,~ b<0$; ~iii)~ $a<0,~ b\geq0$, then there is a solution consisting of a single delta shock for $t\in[0,\ t^{*}]$ and two waves for $t>t^{*}$;
   \item[3)] if the initial data satisfy $[\rho]=0,~ a>0,~ b=-\sqrt{a}$, then there is only a local solution consisting of a single delta shock for $t\in[0,\, t^{*}]$. In addition, $\lim_{t\rightarrow t^*-}x(t)=+\infty$ (resp., $-\infty)$ if $u_0\ge\frac{u_1+u_2}{2}$ (resp., $u_0<\frac{u_1+u_2}{2}$).
 \end{itemize}
\end{theorem}

The proof of item 1) follows immediately from Lemma \ref{lem2.3}. While the proof of item 2) follows from Lemma \ref{lem2.3} for the existence of a delta shock front for $t\in(0,t^*)$. We have $w_\rho(t^*)=0$. The problem for $t\geq t^*$ is reduced to a classical Riemann problem.  
One can  get the solution for $t\geq t^*$  by Theorem \ref{thm3.1}.
For item 3), by Lemma \ref{lem2.3}, we have $\lim_{t\rightarrow t^*-}x(t)=\pm\infty$, hence the existence of the solution is local in time.

Proposition \ref{prp2} shows that it is impossible to construct a delta shock solution with delta shock front separating  piecewise constant states and satisfying the over-compressing entropy condition for all the time. To get a global solution that satisfies \eqref{2.66}, 
  one needs to study a problem such that on at least one side of the delta shock, the flow is not uniform, and
  interactions of rarefaction waves and delta shocks may be involved. This will be reported in other works.

%

\appendix
\section{Wave curves and classical solutions of Riemann problems}

In this appendix, we recall some fundamental results on wave curves and entropy weak solutions to Riemann problems for system \eqref{1.1}, with the initial conditions
\begin{equation}\label{3.1}
(u, \rho)|_{t=0}=\left \{
\begin{split}
&(u_1, \rho_1), \ \ \ \text{if}~x<0, \\
&(u_2, \rho_2), \ \ \ \text{if}~x>0. \\
\end{split} \right.
\end{equation}
The complete analysis and results could be found in \cite[Chapter 2]{Chang1989The}.
It is well-known that \eqref{1.1} has two real eigenvalues
\begin{equation}\label{3.4}\lambda_{1}=u-c, \ \ \ \lambda_{2}=u+c \end{equation}
when $\rho>0,$  where $c=( p'(\rho))^{\frac{1}{2}}.$ So it is strictly hyperbolic for polytropic gases without vacuum. The associated right eigenvectors are
$$\textbf{r}_{1}=\begin{pmatrix}-\rho, & c\end{pmatrix}^{\top}, \ \ \ \textbf{r}_{2}=\begin{pmatrix}\rho,&c \end{pmatrix}^{\top}. $$
Moreover, $\nabla_{(u,\rho)}\lambda_{i}\cdot \textbf{r}_{i}\neq 0$ if $p''(\rho)>0, i=1, 2.$ Hence both the characteristics are genuinely nonlinear for non-vacuum polytropic gases with $\gamma>1$.

For the shock wave curves in the $(u,\rho)$-plane, with $\sigma$ being the speed of shocks, considering the Rankine-Hugoniot conditions
 \begin{equation}\label{3.5}
 \sigma [\rho]=[\rho u],\qquad
 \sigma [\rho u]=[\rho u^{2}+p(\rho)],\end{equation}
which yield
\begin{equation}\label{3.6}
u-u_{1}=\pm \sqrt{\frac{(\rho-\rho_{1})(p-p_{1})}{\rho \rho_{1}}}.
\end{equation}
According to the Lax condition, or equivalently, density increases across shock front \cite[Section 18.B, p.349]{smoller1994shock}, for 1-shocks, we have $\rho>\rho_{1}.$ Since $\sigma<0,$ the first equation in \eqref{3.5} implies that $u<u_{1}.$ Thus
\begin{equation}\label{3.11}
S_{1}\doteq\left\{(u,\rho)~|~  u-u_{1}=-\sqrt{\frac{(\rho-\rho_{1})(p-p_{1})}{\rho \rho_{1}}}, \ \ \ \rho>\rho_{1}\right\}.
\end{equation}
Similar analysis shows that
\begin{equation}\label{3.12}
S_{2}\doteq\left\{(u,\rho)~|~ u-u_{1}=-\sqrt{\frac{(\rho-\rho_{1})(p-p_{1})}{\rho \rho_{1}}}, \ \ \ \rho<\rho_{1}\right\}.
\end{equation}

For the rarefaction wave curves,  set $\xi\doteq\frac{x}{t}.$ Then system \eqref{1.1} is reduced to the ordinary differential equations
\begin{equation}\label{3.14}
\left \{
\begin{split}
&(u-\xi)\rho_{\xi}+\rho u_{\xi}=0,\\
&p'(\rho)\rho_{\xi}+\rho (u-\xi)u_{\xi}=0.\\
\end{split} \right.
\end{equation}
Which implies that
\begin{equation}\label{3.15}
\xi=\lambda_{1}=u-c \quad \text{or} \quad  \xi=\lambda_{2}=u+c.
\end{equation}
For 1-waves, from the first equation in \eqref{3.14}, we have
\begin{equation}\label{3.16} \frac{\mathrm{d} u}{\mathrm{d} \rho}=-\frac{c}{\rho}<0. \end{equation}
Integrating both sides of \eqref{3.16} gives
\begin{equation}\label{3.17}
R_{1}\doteq\left\{(u,\rho)~|~ u-u_{1}=-\int_{\rho_{1}}^{\rho}\frac{\sqrt{p'(s)}}{s}\mathrm{d} s, \ \ \  \rho<\rho_{1}\right\}. \end{equation}
Similarly the 2-rarefaction wave curve is
\begin{equation}\label{3.18} R_{2}\doteq\left\{(u,\rho)~|~ u-u_{1}=\int_{\rho_{1}}^{\rho}\frac{\sqrt{p'(s)}}{s}\mathrm{d} s, \ \ \  \rho>\rho_{1}\right\}. \end{equation}

For the Riemann problem \eqref{1.1}\eqref{3.1} with given left state $(u_{1}, \rho_{1})$, if the right state  $(u_{2}, \rho_{2})$ lies on any of the above four curves, then $(u_{1}, \rho_{1})$ and $(u_{2}, \rho_{2})$ can be connected by a single shock wave or rarefaction wave as indicated by the name of the wave curves above. If $(u_{2}, \rho_{2})$ lies in one of the four open regions $\mathrm{I},\ \mathrm{II},\ \mathrm{III}$ or $\mathrm{IV}$ as depicted in Figure \ref{fig8}, then the solution contains two waves. To be more specific, it is $R_1+R_2$ (two rarefaction waves) if $(u_2,\rho_2)\in \mathrm{I}$; $R_1+S_2$ (a rarefaction wave followed by a shock) if $(u_2,\rho_2)\in \mathrm{II}$; $S_1+R_2$ (a shock followed by a rarefaction wave) if $(u_2,\rho_2)\in \mathrm{III}$; and $S_1+S_2$ (two shocks) if $(u_2,\rho_2)\in \mathrm{IV}$. States in region $\mathrm{V}$ can only be connected to $(u_1,\rho_1)$ by a vacuum lying in the middle of 1- and 2-rarefaction waves.

In Section \ref{s3.1}, for the proof of Theorem \ref{thm3.1}, we utilized the inverse rarefaction wave curves. $R_1^{*}(U_1)$ is exactly the curve given by \begin{align}\label{eq69}\left\{(u,\rho)~|~u-u_{1}=-\int_{\rho_{1}}^{\rho}\frac{\sqrt{p'(s)}}{s}\mathrm{d} s, \ \ \  \rho>\rho_{1}\right\},\end{align} which consists of {\it left state} $U$ that could be connected to the {\it right state} $U_1$ by a 1-rarefaction wave. Similarly, $U$ on the curve
\begin{align}R_2^{*}(U_1)\doteq\left\{(u,\rho)~|~ \label{eq70}u-u_{1}=\int_{\rho_{1}}^{\rho}\frac{\sqrt{p'(s)}}{s}\mathrm{d} s, \ \ \  \rho<\rho_{1}\right\}\end{align}
could be connected from right by a 2-rarefaction wave to the right state $U_1$.
\begin{figure}[t]
\centering
  \includegraphics[width=0.5\textwidth]{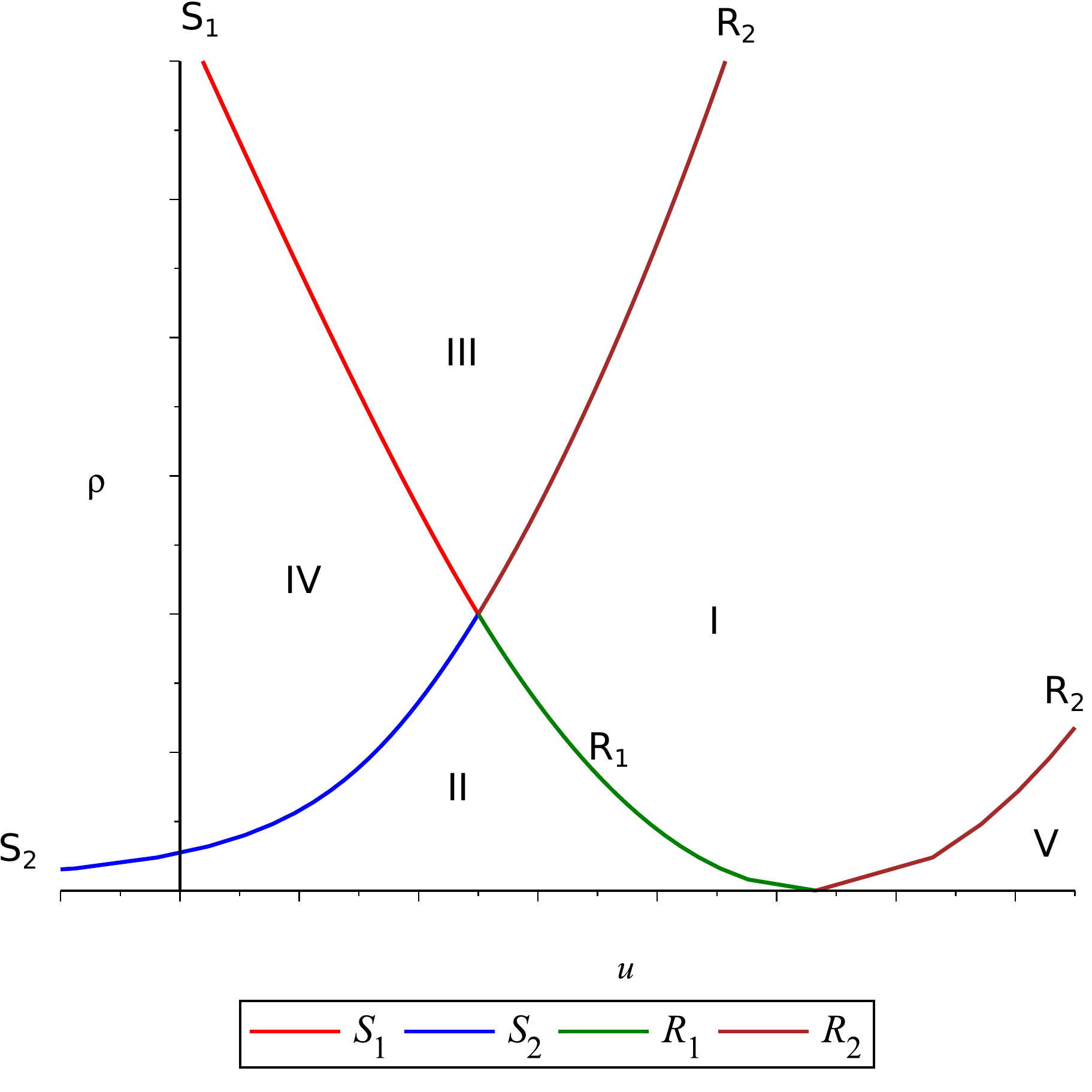}
\caption{\small Wave curves of Riemann problem \eqref{1.1}\eqref{3.1} in the $(u,\rho)$-plane.}\label{fig8}
\end{figure}

\begin{acknowledgements}

The authors are grateful to Professor Jiequan Li, for his valuable comments on a draft of this paper, and particularly the observation on connections between delta shocks and free pistons, in a private conversation. Aifang Qu appreciates very much the support and the hospitality of the IMS,  during her visit at the Chinese University of Hong Kong.

\end{acknowledgements}

%
 \section*{Conflict of interest}
The authors declare that they have no conflict of interest.



\end{document}